\documentclass[reqno, 11pt]{amsart}

\usepackage{fullpage} 
\usepackage{calc,amsfonts,amsthm,amscd,epsfig,psfrag,amsmath,amssymb,enumerate,graphicx}

\usepackage[colorlinks=true, pdfstartview=FitV, linkcolor=blue, citecolor=blue, urlcolor=blue,pagebackref=false]{hyperref}

\usepackage{srcltx}
\numberwithin{equation}{section}
\usepackage{enumerate}
\newcommand{\tc}{\, | \,}
\newcommand\wt{\widetilde}

\newcommand\bbF {{\mathbb F}}
\newcommand\bbG {{\mathbb G}}
\newcommand\cL{\mathcal L}

\newcommand\cC{\mathcal C}

\newcommand\Var{{\rm {Var}}}

\newcommand\cE{\mathcal E}

\newcommand\tmix{{\rm T_{mix}}}

\newcommand\bbN {{\mathbb N}}
\newcommand\bbR {{\mathbb R}}
\newcommand\bbZ {{\mathbb Z}}
\renewcommand\O {\Omega}
\newcommand\bbP{{\mathbb P}}

\newcommand\gap {{\rm gap}}

\newcommand\RR {{\mathbb R}}

\newcommand\ZZ {{\mathbb Z}}

\usepackage{amssymb, amsmath}
\newtheorem{rem}{Remark}
\newtheorem{claim}{Claim}
\newtheorem{step}{Step}

\newtheorem{prop}{Proposition}
\newtheorem{theo}{Theorem}
\newtheorem{lemma}{Lemma}
\newtheorem{definition}{Definition}

\renewcommand{\L}{\Lambda}
\renewcommand{\l}{\lambda}
\renewcommand{\d}{\delta}

\begin{document}

\title[]{Mixing times of monotone surfaces and SOS interfaces: a mean curvature %motion 
approach}
\author[P.~Caputo]{Pietro Caputo}
\address{P. Caputo, Dipartimento di Matematica,
  Universit\`a Roma Tre, Largo S.\ Murialdo 1, 00146 Roma, Italia.
  e--mail: {\tt caputo@mat.uniroma3.it}}
\author[F.~Martinelli]{Fabio Martinelli} 
\address{F. Martinelli, Dipartimento di Matematica,
  Universit\`a Roma Tre, Largo S.\ Murialdo 1, 00146 Roma, Italia.
  e--mail: {\tt martin@mat.uniroma3.it}} 
\author[F.L.~Toninelli]{Fabio Lucio Toninelli} 
\address{F. L. Toninelli, CNRS and ENS Lyon, Laboratoire de Physique\\
  46 All\'ee d'Italie, 69364 Lyon, France.  e--mail: {\tt
    fabio-lucio.toninelli@ens-lyon.fr}}

\thanks{This work was
  supported by the European Research Council through the ``Advanced
  Grant'' PTRELSS 228032}

\begin{abstract}
  We consider stochastic spin-flip %natural Glauber 
  dynamics for: (i) monotone discrete
  surfaces in $\bbZ^3$ with planar boundary height and (ii) the
  one-dimensional discrete Solid-on-Solid (SOS) model confined to a box. In both cases
  we show almost optimal bounds $O(L^2\text{polylog}(L))$ for the mixing time of the
  chain, where $L$
  is the natural size of the system. %As it is well known (see e.g. \cite{??}), in the above
  %situations t
  The dynamics at a macroscopic scale should be described
  by a deterministic mean curvature motion such that each point of the
  surface feels a drift which tends to minimize the local surface
  tension \cite{Spo}. Inspired by this heuristics, our approach consists in
  bounding the dynamics with an auxiliary one which, with very high
  probability, follows quite closely the deterministic mean curvature
  evolution. Key technical ingredients are monotonicity, coupling and an
  argument due to D. Wilson \cite{Wilson} in the framework of lozenge
  tiling Markov Chains.  Our approach works equally well for both
  models despite the fact that their equilibrium maximal height fluctuations
  occur on very different scales (${\log L }$ for monotone surfaces and
  $\sqrt L$ for the SOS model).  Finally, combining techniques %borrowed
  from kinetically constrained spin systems \cite{CMRT} together with
  the above mixing time result, we prove an almost diffusive lower
  bound of order $1/L^{2}\text{polylog}(L)$ for the spectral gap of the SOS
  model with  horizontal size $L$ and unbounded heights.
\end{abstract}

\maketitle

\thispagestyle{empty}
\tableofcontents
\section{Introduction}
Understanding the dynamical behavior of interfaces undergoing a
stochastic microscopic evolution and the emergence
of mean curvature motion on the macroscopic scale is a fundamental
problem in non-equilibrium statistical mechanics \cite{Spo,Funaki}.
Even the simpler question of rigorously establishing the correct time-scale for
the relaxation to equilibrium is in many cases a
challenge. Similar questions arise also in combinatorics and computer
science when the interface configurations can be put in correspondence
with the dimer coverings of a planar graph: the main focus there is to evaluate the running time of Markov
Chain algorithms which sample uniformly among such
combinatorial structures (cf. in particular \cite[Section 5]{Wilson} for background and
motivations in this direction). In this
paper we address this question for two natural and widely studied
models
and obtain essentially optimal bounds on the equilibration time.

The first classical example is the continuous-time single spin-flip dynamics of discrete monotone surfaces
with
fixed boundary, cf.\ for instance
\cite{LRS,Wilson}. Monotone surfaces can be visualized as a stack of unit
cubes centered around the vertices of $\bbZ^3$, which is decreasing in
both the $x$  and the $y$ direction (see Figure \ref{mon-surf}); see 
Section \ref{sec:3d} for the precise definition. 
\begin{figure}[h]
\centerline{
\includegraphics[scale=1]{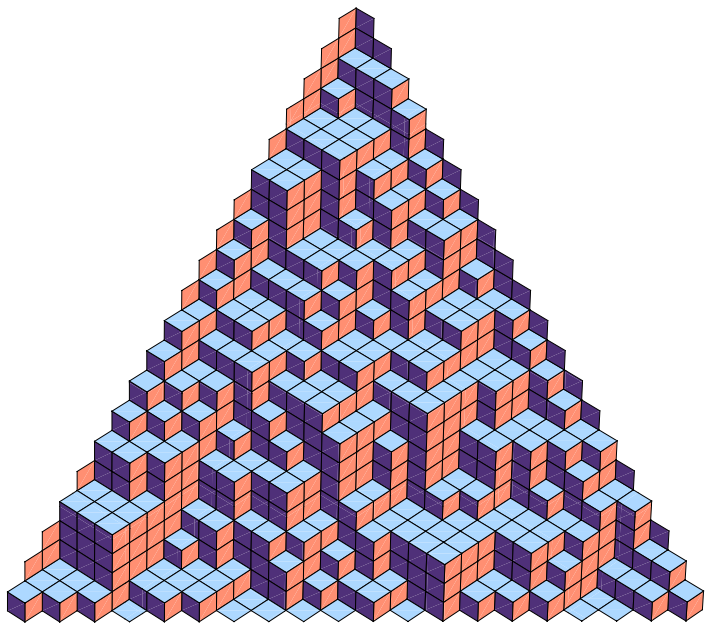}}
\caption{An example of a monotone surface with planar boundary
  conditions on the plane $x+y+z=0$. The picture is taken from
  \cite{Kenyon-lecturenotes}.}
\label{mon-surf}
\end{figure}
This is equivalent to the zero-temperature
dynamics of the three-dimensional Ising model with boundary conditions
that enforce the presence of an interface;  alternatively, it can be seen as a stochastic dynamics for 
lozenge tilings of a finite region of the plane, or of  dimer
coverings of a finite region of the honeycomb lattice. 
The invariant measure is uniform over all discrete surfaces compatible with the monotonicity 
constraints and with the boundary conditions.
The conjectured behavior of the 
mixing time $\tmix$ is of order $L^2\log L$ if $L$ is the linear size
of the region under consideration. The monotonicity
condition produces dynamical constraints which prevent the application
of standard tools to obtain non-trivial 
bounds on $\tmix$. To overcome this difficulty,  
a modified ``non-local'' version of the dynamics, whose moves involve adding or removing piles of unit cubes stacked one on top of the other, was introduced in 
\cite{LRS}. Using this device, a polynomial (in $L$) upper bound on $\tmix$ was
proven in \cite{LRS}. An important breakthrough was obtained by
D. Wilson in \cite{Wilson},
where the mixing time of the non-local dynamics was sharply analysed
and shown to be of order $L^2\log L$ (from below and above). Via
classical comparison arguments, this implies \cite{RandallTetali} that
$\tmix =\wt O(L^6)$ for the single-site dynamics - here and below we use the notation $\wt O(L^p)$ for any quantity that is bounded above by $L^p$ up to polylog($L$) factors. An improved
comparison, relying also on the so-called Peres-Winkler censoring
inequalities \cite{notePeres}, shows that $\tmix = \wt O( L^4)$; see \cite[Section 4.1]{CMST}.

The second example is the continuous-time stochastic one-dimensional 
SOS model, described by a set of integer-valued heights
$\eta_1,\ldots,\eta_L$ such that each $\eta_i$ is confined in an
interval whose size is of order $L$ (see Section \ref{sec:SOS} for the
precise definition and Figure \ref{fig:sos} for an illustration)
and the heights $\eta_0,\eta_{L+1}$ are fixed boundary conditions.
\begin{figure}[h]
\centerline{
\includegraphics[scale=1,width=10cm]{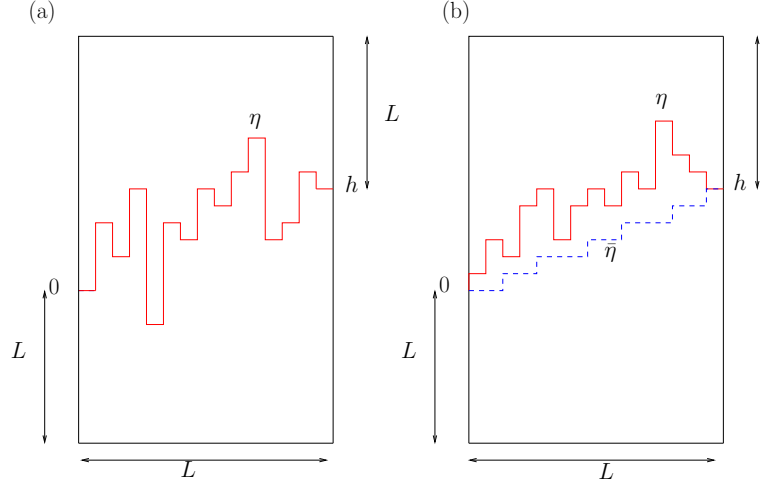}}
\caption{(a) A SOS configuration $\eta$ on the interval $[1,L]$, with boundary
  conditions $0$ and $h$ and heights confined between $-L$ and $L+h$
(proportions are distorted for graphical convenience). (b) A SOS
configuration 
constrained to lie above the ``wall'' $\bar \eta$, defined as $\bar\eta_i = \lfloor ih/(L+1)\rfloor$
(i.e.\ the integer approximation of the straight line joining the two boundary heights).}
\label{fig:sos}
\end{figure}
The
dynamics involves local moves where the discrete heights can
change by $\pm1$ at each step. 
The invariant measure is the Gibbs distribution corresponding to a potential given by the absolute value of the height
gradients. 
Due to the non-strictly convex character of
the interaction, even obtaining a diffusive spectral gap bound (of
order $L^{-2}$) for zero boundary conditions has been  a long-standing open problem. The analysis of an auxiliary
non-local dynamics, in the spirit of \cite{Wilson}, plus a judicious use
of the 
Peres-Winkler inequalities were recently combined to obtain a mixing time
upper bound $\wt O(L^{5/2})$ 
\cite{MS},
while again the conjectured behavior is $O(L^2\log L)$. 

The main contribution of the present paper is a proof that, for both
problems, $\tmix = \wt O(L^2)$; 
see Theorem \ref{th:3d} for monotone surfaces and Theorem \ref{th:sos} for SOS. In the monotone surface case, we must restrict our analysis  to the case where the boundary condition is \emph{approximately planar} (see e.g. Figure \ref{mon-surf}), so that on a
macroscopic scale, at equilibrium, the surface is flat, cf. Theorem \ref{th:eq3d}. At the microscopic scale this corresponds to the zero temperature limit of the 3D Ising model with so-called Dobrushin boundary conditions, i.e.\ the boundary spins take the value $\pm1$ according to whether they lie above or below a given plane. 
In any finite volume, this leads to the uniform distribution over all monotone surfaces
compatible with the given ``planar'' boundary conditions. In the
infinite volume limit 
$L\to\infty$, such uniform distribution is related to  a
\emph{translation invariant, ergodic Gibbs measure} on dimer
coverings of the infinite honeycomb lattice
\cite{KOS}, described by a 
determinantal point processes whose kernel is known explicitly. Our
method relies crucially on the Gaussian Free Field-like fluctuation
properties of such infinite volume state.
 
Theorem \ref{th:sos}, together with techniques
developed in the context of kinetically constrained spin models \cite{CMRT}, allows us to implement a recursive analysis
whose output is an almost diffusive lower bound $\gap\ge 1/L^{2}$ up to polylog($L$) factors for the spectral gap of the SOS model with \emph{unbounded} heights
and zero boundary conditions $\eta_0=\eta_{L+1}=0$; see Theorem \ref{th:gap}.

Our approach can be roughly described as follows. At equilibrium the interface is macroscopically flat,
with maximal height fluctuations much smaller than $L$ (logarithmic in the monotone
surface case and of order $\sqrt L$ for SOS). The main step in the proof of the mixing time upper bounds
is to show that an initial macroscopically non-flat profile approaches
the flat equilibrium profile within the correct time. Heuristically, the
interface evolves by minimizing the surface tension and
therefore feels a drift proportional to the local mean curvature.
A major difficulty encountered in previous approaches is the possible
appearance of very large (up to  order $L$) gradients of the surface
height. Our method consists in introducing an auxiliary ``mesoscopic''
dynamics which approximately follows the mean curvature
flow, such that large gradients are absent in the initial
condition and are very unlikely to be created at later times. The
key point is that, thanks to monotonicity and coupling considerations, the original
dynamics converges to equilibrium faster than the auxiliary one.

The study of the mesoscopic dynamics involves a local analysis of the relaxation of a
mesoscopic portion of the interface, whose size $\ell$ depends i) on the
current local mean curvature $1/R_t$ of the surface  at time $t$ and ii) on the size of
equilibrium height  fluctuations. For the monotone
surface case it turns out that one has to choose $ \ell\simeq
R_t^{1/2}$ and for the SOS model $\ell\simeq
R_t^{2/3}$
(in both cases modulo polylogarithmic factors). 
We refer to Section
\ref{sec:scale} for a more detailed explanation of the strategy which
leads to the choice of the different
scales in the two models. Note that, since the
evolution tends to a profile with vanishing mean curvature, $R_t$ grows
with time and becomes much larger than the initial value $R_0\simeq L$
when equilibrium is approached.
One key input is to prove that the equilibration time within such
mesoscopic regions is of the correct order $\wt O(\ell^2)$:
this result can be obtained following recent important progress in \cite{CMST} for monotone
surfaces and in \cite{MS} for SOS (both use the Peres-Winkler
inequalities and Wilson's argument \cite{Wilson} in an essential way).
The other key point, which is the main contribution of this work, is to
show that the mesoscopic dynamics is
dominated with high probability by
the deterministic mean curvature evolution of a macroscopic smooth
profile with the appropriate boundary conditions.

In this paper we do not focus on mixing time \emph{lower bounds}. Let us
however mention that, for the monotone surface dynamics, under natural
assumptions mentioned in Remark \ref{rem:natural} below, one can apply
\cite[Theorem 3.1]{CMST} to obtain 
a lower bound of order $L^2/\log L$ for the mixing time.
Similarly, for the SOS model it is known that the spectral gap is at
most of order $L^{-2}$ \cite{MS}, which by standard inequalities 
%due to the mentioned
%inequality $\gap\ge 1/\tmix$ 
directly implies a lower bound of order
$L^2$ for $\tmix$.

We believe our method could potentially work for a wide range of
stochastic interface dynamics models where mean curvature motion is
expected to occur macroscopically.
An example that comes naturally to mind is the dynamics of domino
tilings of the plane, for which at present only non-optimal polynomial
upper bounds on $\tmix$ are available \cite{LRS}.

A real challenge is on the other hand to prove that $\tmix = \wt O( L^2)$ for the
monotone surface dynamics when the boundary height is \emph{not}
approximately planar, in which case the equilibrium shape is not
macroscopically flat and arctic circle-type phenomena can occur \cite{CLP}. While in principle our idea of mesoscopic
auxiliary dynamics could be adapted to this case too, what is missing
here are precise, finite-$L$ equilibrium estimates on height fluctuations and on
the rate of convergence of the equilibrium average height to its macroscopic limit.

Concerning the SOS model, a big challenge is the analysis of  the
dynamics for the model in dimension $2+1$, for which not even crude
polynomial bounds on $\tmix$ are available, while on general grounds
one can expect\footnote{This model is actually
quite tricky and can hide surprises: one can for instance show \cite{CLST} that, if one adds
a hard-wall floor constraint at height zero, the boundary
height being also fixed to zero, the dynamics is slowed down by the
presence of a bottleneck which causes the relaxation time to be
exponentially large in $L$. This phenomenon, related to the so-called entropic repulsion,  should not be present in
absence of the wall.}  once more the ``diffusive'' scaling $L^2\log L$.

\subsection{Generalities and notation}
Let us recall some standard  definitions for continuous
time reversible Markov chains (see e.g. \cite{Peres}). We will mostly work in the case
where the state space $\Omega$ is finite and  the Markov chain is irreducible. In particular,
there is a unique reversible invariant measure
$\pi$. For $\xi\in\Omega$ and $t\ge0$, 
%$\sigma^\xi_t$ denotes the configuration at time $t$ started from the
$\mu^\xi_t$ denotes the law of the configuration at time $t$ started from the
initial configuration $\xi$. % The law of $\sigma^\xi_t$
% is denoted $\mu^\xi_t$ and 
The law of the chain is denoted by  
$\bbP$.

Given two laws $\mu,\nu$ on $\Omega$, we let $\|\mu-\nu\|=\sup_{A\subset\Omega}|\mu(A)-\nu(A)|$ denote
their total variation distance.
The {\em mixing time} $\tmix$, defined as 
\begin{eqnarray}
  \label{eq:3}
  \tmix=\inf\big\{t>0: \sup_{\xi\in\Omega}\|\mu^\xi_t-\pi\|\le (2e)^{-1}\big\},
\end{eqnarray}
measures the time it takes for the dynamics to be close in total variation
to equilibrium, uniformly in the initial condition.
It is well known that
\begin{eqnarray}
  \label{eq:10}
  \sup_{\xi\in\Omega}\|\mu^\xi_t-\pi\|\le e^{-\lfloor t/\tmix\rfloor},
\end{eqnarray}
i.e., the worst-case variation distance from equilibrium
decays exponentially with rate $1/\tmix$.

\smallskip

If $\cL$ denotes the infinitesimal generator of the reversible Markov chain, the {\em spectral gap} 
is defined as the lowest nonzero eigenvalue of $-\cL$. Equivalently, if $\cE(f,g) = \pi[f(-\cL g)]$
denotes the associated Dirichlet form, one has
\begin{equation}\label{defgap}
\gap = \inf_f \frac{\cE(f,f)}{\Var(f)}\,,
\end{equation} 
where $\Var(f)$  stands for the variance $\pi[f^2]-\pi[f]^2$ and the infimum ranges over all functions $f:\Omega\mapsto\bbR$ such that $\Var(f)\neq 0$. This definition makes sense also in the case where $\O$ is countably infinite, as will be the case for the unbounded SOS model to be considered in Theorem \ref{th:gap} below.
\smallskip

Throughout the paper, we will adopt the following conventions:
\begin{enumerate}[(i)]
\item if $x,y\in \RR^n$, then $d(x,y)$ denotes their Euclidean
  distance;

\item if $x\in \RR^n$, then we write $x^{(a)},a=1,\ldots,n$ for its components;
\item if $U\subset \RR^n$, then $diam(U)=\max\{d(x,y), x,y\in U\}$
  denotes its diameter;

\item if $U\subset \ZZ^n$, then $\partial U=\{x\in
 \ZZ^n\setminus U\mbox{\;such that\;} \exists y\in U \mbox{\;with\;} d(x,y)=1\}$.
If on the other hand $U$ is a smooth subset of $\RR^n$, then $\partial
U$ denotes its usual boundary.
\end{enumerate}

\section{Monotone surfaces with ``planar'' boundary conditions}

\label{sec:3d}

\begin{definition}[Monotone surfaces]
%   A subset $V\subset \ZZ^3$ is a (discrete) monotone set if $z\in V$ implies
%   $y\in V$ whenever $y^{(a)}\le z^{(a)}, a=1,2,3$. The collection of
%   all monotone sets is denoted $\Sigma$.\\
 A function $\phi:\ZZ^2\mapsto (\ZZ\cup\{\pm\infty\})$
  defines a (discrete) monotone surface if $\phi_x\ge
  \phi_y$ whenever $x^{(a)}\le y^{(a)},a=1,2$.  The collection of
  all monotone surfaces is denoted by $\Omega$.
\end{definition}

On % $\Sigma$ (resp. on 
$\Omega$ there is a natural partial order: we say
that % $V_1\preceq V_2$ (resp. 
$\phi\le \phi'$ if % $V_1\subset
% V_2$ 
% (resp.
$\phi_x\le \phi'_x$ for every $x\in\ZZ^2$.
Analogously, for $U\subset \ZZ^2$ we write $\phi_U\le
\phi'_U$ if $\phi_x\le \phi'_x$ for every $x\in U$. 
% There is an obvious one-to-one correspondence between monotone sets
% and monotone surfaces: given $V\in \Sigma$ one defines $\phi\in\Omega$
% by
% \[
% \ZZ^2\ni x\mapsto \phi(x):=\sup\{n\in \ZZ:(x^{(1)},x^{(2)},n)\in V\}.
% \]
% \begin{definition}
% \label{def:vphi}
%   Given $\phi\in\Omega$ (resp. $V\in \Sigma$) we denote by $V_{[\phi]}$
%   (resp. $\phi_{[V]}$) the corresponding monotone set (resp. surface).
% \end{definition}

% Given a finite set $U\subset \ZZ^2$, let \[\partial U=\{x\in
% \ZZ^2\setminus U\mbox{\;such that\;} \exists y\in U \mbox{\;with\;} d(x,y)=1\}.\]
% The boundary conditions $\eta$ are an assignment
% \[
% \eta:\partial U\mapsto (\ZZ\cup\{\pm\infty\})
% \]
% with the constraint that there exists $\phi\in\Omega$ such that
% $\phi(x)=\eta(x)$ for
% $x\in \partial U$. In other words, $\eta$ must be the restriction to
% $ \partial U$ of some monotone surface $\phi_\eta$..

\subsection{Heat bath dynamics}
We define a dynamics $\{\phi^\xi(t)\}_{t\ge0}$ on monotone surfaces
with initial condition $\xi$ and  fixed boundary
conditions (b.c.) outside a finite region. Let $U$ be a finite connected subset of $\ZZ^2$ (the finite region) and 
$\eta\in \Omega$ (the boundary condition). Without loss of generality we will always assume that $U$ contains the origin.

Given $\phi\in \Omega$ and $x\in \ZZ^2$, let $\phi^{(x,+)},\phi^{(x,-)}\in\Omega$ be
  defined by
  \begin{eqnarray}
    \label{eq:1}
    \phi^{(x,+)}_y=\left\{
      \begin{array}{lll}
        \phi_y & if & x\ne y\\
        \min[\phi_x+1,\phi_{(x^{(1)}-1,x^{(2)})},\phi_{(x^{(1)},x^{(2)}-1)}] & if &x=y
      \end{array}
\right.
\end{eqnarray}
and
\begin{eqnarray}
\label{eq:phi-x}
    \phi^{(x,-)}_y=\left\{
      \begin{array}{lll}
        \phi_y & if & x\ne y\\
        \max[\phi_x-1,\phi_{(x^{(1)}+1,x^{(2)})},\phi_{(x^{(1)},x^{(2)}+1)}] & if &x=y.
      \end{array}
\right.
  \end{eqnarray}

 The dynamics 
is a continuous-time Markov chain on the set
\[
\Omega_{\eta,U}:=\{\phi\in\Omega:\phi_x=\eta_x \mbox{\;for\;}x\notin U\}.
\] 
The initial condition at time zero is some given
$\xi\in\Omega_{\eta,U}$.  To each $x\in U$ is assigned an
i.i.d. exponential clock of rate $1$. If the clock labeled $x$ rings
at time $t$, we replace $\phi(t)$ with $[\phi (t)]^{(x,+)}$ or
$[\phi (t)]^{(x,-)}$ with equal probabilities. It is immediate to check
that such Markov chain is irreducible and reversible with respect to the uniform
measure on $\Omega_{\eta,U}$, which we denote $\pi^{\eta}_U$ or simply $\pi$.
The mixing time is then defined as in \eqref{eq:3} where the supremum
is taken over $\xi\in \Omega_{\eta,U}$.

\subsection{Monotonicity}
\label{sec:monoton}
A function $f$ on
$\Omega$ is said to be increasing (resp. decreasing) if $f(\phi)\le
f(\phi')$ (resp. $f(\phi)\ge
f(\phi')$) whenever $\phi\le \phi'$. Given two laws $\mu,\nu$ on
$\Omega$, we write $\mu\preceq \nu$ ($\nu$ dominates stochastically
$\mu$) if $\mu(f)\le \nu(f)$ for every increasing function $f$.
The heath-bath dynamics is monotone (or attractive) with respect to
the partial ordering ``$\le$'', in the following sense. If $\mu^\xi_{t,\eta}$
denotes the law of $\phi^\xi_{\eta}(t)$, the dynamics at time $t$ started from $\xi$ and
evolving with b.c. $\eta$,  one has the following property (cf. for
instance the discussion in \cite[Sec. 2.1]{CMST}):
\[
\mu_{t,\eta}^\xi\preceq \mu_{t,\eta'}^{\xi'} \mbox{\;if\;}\xi\le \xi'\mbox{\;and\;}\eta\le\eta'.
\]
In particular, letting $t\to\infty$, one has $\pi_U^{\eta}\preceq \pi_U^{\eta'}$.
It is possible to realize on the same probability space the trajectories of the
Markov chain corresponding to distinct initial conditions $\xi$ and/or distinct boundary
conditions $\eta$ in such a way that, with probability one,
\[
\phi^\xi_{\eta}(t)\le \phi^{\xi'}_{\eta'}(t)\mbox{
\; for every\;} t \ge 0,\mbox{\; if\;}\xi\le \xi'\mbox{\;and\;}\eta\le \eta'.
\]
Such a construction takes the name of \emph{global monotone coupling}.
Throughout the paper we will apply several times the above monotonicity properties:
for brevity, we will simply say ``by monotonicity...''

\subsection{Mixing time upper bound}

As we mentioned in the introduction, it is expected that 
$\tmix=O(L^2\log L)$, where $L$ is the diameter of the region $U$. The next
result proves such conjecture, up to logarithmic corrections, under
the assumption that the boundary conditions are ``approximately
planar'' (cf. condition \eqref{eq:4} below). Such ``planar'' case is
rather natural in terms of the three-dimensional Ising model: indeed,
it corresponds to the zero-temperature limit of a system defined in
the cylinder $U\times \ZZ$, with Dobrushin-type boundary conditions which are, say,
``$+$'' above some plane and ``$-$'' below.

\begin{definition}
  \label{def:nu}
Given ${\bf n}\in\RR^3$ with $\|{\bf n}\|=1$, we write ${\bf n}>0$ if 
${\bf n}^{(i)}>0,i=1,2,3$.
We let $\bar \phi^{\bf n}\in\Omega$ be the discrete monotone
surface with slope ${\bf n}$:
\[
\ZZ^2\ni x\mapsto\bar \phi^{\bf n}_x=\max\{z\in \ZZ:x^{(1)}{\bf n}^{(1)}+x^{(2)}{\bf
  n}^{(2)}+z{\bf n}^{(3)}\le 0\}
\]
 and $\Pi^{\bf n}$ denotes the plane 
\begin{eqnarray}
  \label{eq:6}
  \Pi^{\bf n}=\{z\in\RR^3: {\bf n}\cdot z=0\}.
\end{eqnarray}
For $x\in \ZZ^2$ we let $\Pi^{\bf n}(x)$ denote the vertical coordinate of the
point in the plane $\Pi^{\bf n}$ with horizontal coordinates $(x^{(1)},x^{(2)})$.
For $h>0$ (resp. $h<0$), $\Pi^{\bf n}_h$ is the plane obtained
translating $\Pi^{\bf n}$ upwards
(resp. downwards) by $|h|$ along the vertical direction.

\end{definition}
The planarity condition on the boundary conditions is specified  as follows:
\begin{definition}\label{ass:eta}
   Let ${\bf n}>0$. We say that $\eta$ is a \emph{good planar boundary condition} with slope $\bf n$ if there exists $C>0$ such that 
  \begin{eqnarray}
    \label{eq:4}
|\eta_x-\bar\phi^{\bf n}_x|\le C\log (|x|+1)    
  \end{eqnarray}
for every $x\in \bbZ^2$.
\end{definition}
As we mentioned in the introduction, under such boundary conditions the surface at equilibrium is essentially flat:
\begin{theo}
\label{th:eq3d}
Let ${\bf n}>0$ and let $\eta$ be a good planar boundary condition with slope $\bf n$. For every $\epsilon>0$ there exists $c>0$ such that for every $a>0$ the following holds. Let $U$ be a finite, connected subset of $\bbZ^2$ containing the origin and let $L:=diam(U)$. Then, for any $L$ large enough,
 \begin{eqnarray}
    \label{eq:17}
    \pi^{\eta}_U\left(\exists y\in U:
|\phi_y-\bar \phi^{\bf n}_y|\ge a(\log L)^{1+\epsilon}\right)\le(1/c) e^{-c\,a(\log L)^{1+\epsilon}}.
  \end{eqnarray}
\end{theo}
We can finally formulate our mixing time upper bound:
\begin{theo}
\label{th:3d} In the same assumption of Theorem \ref{th:eq3d}, for $L$ sufficiently large one has
\begin{eqnarray}
  \label{eq:2}
  \tmix\le L^2(\log L)^{12}.
\end{eqnarray}
\end{theo}
With some technical effort (but no need of new ideas) one can improve
the exponent $12$ to $6$ but we
will not do so, since neither is close to the conjectured optimal value $1$.
\begin{rem}
  \label{rem:natural}
Concerning lower bounds: 
using the same idea of the proof of the lower bound on the mixing time of the
three-dimensional zero-temperature Ising model with ``$+$'' boundary
conditions  in
\cite[Theorem 3.1]{CMST}, it is not hard to see that $\tmix\ge L^2/(c\log L)$
for a suitable $c>0$ for instance when $U= \mathcal U_L\cap \ZZ^2$, with
$\mathcal U_L$ a smooth open set of $\RR^2$ expanded by a factor $L$.
\end{rem}
\begin{rem}
Concerning the assumption ${\bf n}>0$ a first obvious observation is
that, using lattice symmetries, one could replace it with the
condition ${\bf n}^{(i)}\neq 0$, $i=1,2,3$. The main reason why we
excluded the case in which one or two components of $\bf n$ vanish is
that, in these cases,  the fluctuations of the surface in $U$ around
$\bar \phi^{\bf n}$ are \emph{
deterministically} upper bounded by $C'\log L$, where $C'$ depends on the constant $C$ in \eqref{eq:4}. As a consequence Theorem \ref{th:eq3d} becomes trivial and one can appeal to Proposition \ref{prop:wilson} below to get immediately Theorem \ref{th:3d}. Thus, the really interesting and non-trivial case is ${\bf n}>0$.
\end{rem}

\subsection{Dynamics with ``floor'' and ``ceiling''}
\label{sec:flocei} In the course of the proof of Theorem \ref{th:3d} we need an auxiliary restricted dynamics for an interface constrained between a floor and a 
ceiling.
Let $U$ and $\eta$ be as in the previous section; fix some
$\phi^+,\phi^-\in\Omega_{\eta,U}$ with $\phi^-\le \phi^+$ and let 
\begin{eqnarray}
  \label{eq:9}
  \Omega_{\eta,U}^{\phi^\pm}=\{\phi\in\Omega_{\eta,U}: \phi^-\le \phi\le \phi^+\}.
\end{eqnarray}

 One can
define a dynamics restricted to $\Omega_{\eta,U}^{\phi^\pm}$ simply by choosing an initial
condition $
\xi\in \Omega_{\eta,U}^{\phi^\pm}$ 
and redefining 
  \begin{eqnarray}
   \phi^{(x,+)}(y)=\left\{
      \begin{array}{lll}
        \phi(y) & if & x\ne y\\
        \min[\phi_x+1,\phi_{(x^{(1)}-1,x^{(2)})},\phi_{(x^{(1)},x^{(2)}-1)},\phi^+_x] & if &x=y
      \end{array}
\right.
\end{eqnarray}
and
\begin{eqnarray}
    \phi^{(x,-)}(y)=\left\{
      \begin{array}{lll}
        \phi(y) & if & x\ne y\\
        \max[\phi_x-1,\phi_{(x^{(1)}+1,x^{(2)})},\phi_{(x^{(1)},x^{(2)}+1)},\phi^-_x] & if &x=y
      \end{array}
\right.
  \end{eqnarray}
(compare with Eqs. \eqref{eq:1}, \eqref{eq:phi-x}). The dynamics is
again monotone in the sense of Section \ref{sec:monoton}, but this time
the invariant measure $\pi$  is the uniform measure on $\Omega_{\eta,U}^{\phi^\pm}$.

\section{%Mixing time and spectral gap bounds for the 
Solid-on-Solid model}
\label{sec:SOS}
We turn to the study of the mixing time and spectral gap of a
one-dimensional interface of Solid-on-Solid (SOS) type. 
The generic configuration (height function) of the standard SOS model is
$\eta=(\eta_1,\dots,\eta_L)\in\ZZ^L$ and its equilibrium measure $\wt\pi=\wt\pi_{L,h}$ corresponding to 
boundary conditions $0,h\in\bbZ$
is 
\begin{equation}
\label{eq:unbo}
\wt\pi_{L,h}(\eta) \propto \exp{\left(-%\beta
\sum_{i=0}^{L}|\eta_{i+1}-\eta_{i}|\right)}
\end{equation}
with $\eta_0=0$ and $\eta_{L+1}=h$.
%The measure $\pi_{L,h}$ describes the equilibrium 
%of a line interface of SOS type with left boundary condition $0$ and
%right boundary condition $h$
%(the most commonly studied case being $h=0$). 
There is no inverse temperature parameter $\beta$ in \eqref{eq:unbo}
since in this one-dimensional model its numerical value does not
affect the qualitative behavior of the system and there is no loss of
generality in fixing its value to unity.
It is well known that $\wt\pi_{L,h}$ describes the law of the unique open contour in
the two-dimensional Ising model in the box $\{1,\ldots,L\}\times\bbZ$ with Dobrushin boundary conditions (boundary
spins are ``$+$'' under the line which joins $(0,0)$ to $(L+1,h)$ and
``$-$'' below it), in
the limit where the couplings on vertical edges tend to infinity.

Since 
the mixing time \eqref{eq:3} deals with relaxation to equilibrium from
an arbitrary initial condition it is necessary to introduce the
following {\em bounded} version \cite{MS} of the SOS model,
enclosed in a rectangular box of sides of order $L$. Thanks to
standard equilibrium estimates, see also Lemma \ref{lem:eq} below, the
behavior at equilibrium of this bounded version of the model is
essentially the same as the usual unbounded one defined above. We come back to the unbounded model in Theorem \ref{th:gap}, which deals with the spectral gap.

For nonnegative integers $L$ and $h\leq L$, consider the configuration space $\O_{L,h}$ defined by
\begin{equation}
\label{OLh}
\O_{L,h} = \left\{\eta=(\eta_1,\dots,\eta_L)\,,\; \eta_i\in\bbZ\cap[-L,L+h]\,
\right\}\,.
\end{equation}
The equilibrium measure on $\O_{L,h}$ is then given by $\pi=\pi_{L,h}=\wt\pi(\cdot\tc\O_{L,h})$. 
%$\pi=\pi_{L,h}$, with 
%\begin{equation}\label{muLh}
%\pi_{L,h}(\eta) = \frac{\exp{(-%\beta
%\sum_{i=0}^{L}|\eta_{i+1}-\eta_{i}|)}}{Z_{L,h}}\,,\qquad
%Z_{L,h}=\sum_{\eta\in\O_{L,h}}\exp{\Big(-%\beta
%\sum_{i=0}^{L}|\eta_{i+1}-\eta_{i}|\Big)}\,,
%\end{equation}
%where %$\beta>0$, 
%$\eta_0=0$ and $\eta_{L+1}=h$.   Thus,  is just \eqref{eq:unbo} restricted to the set $\O_{L,h}$.
%The parameter $\beta>0$ determines the strength of the interaction,
%however in this one-dimensional model its numerical value does not
%affect the qualitative behavior of the system. In particular, there
%will be no loss of generality in fixing its value to $\beta \equiv 1$
%in the sequel. %Figure ?
 % in the sequel 

Occasionally we will consider the SOS model with further hard wall constraints, obtained by conditioning $\pi_{L,h}$ to the event $\xi^1\leq \eta\leq\xi^2$, where $\xi^i\in\O_{L,h}$, $i=1,2$ are two configurations such that $\xi^1\leq \xi^2$. Here, and below, we use the notation $\xi\leq \sigma$, for the natural partial order in $\O_{L,h}$ defined via 
$\xi_i\leq \sigma_i$, for all $i=1,\dots,L$. We refer to $\xi^1,\xi^2$ as the floor and the ceiling, respectively, and write $\pi_{L,h}^{\xi^1,\xi^2}$ for the corresponding equilibrium measure.
%Another version of this model is obtained by imposing the constraint that 
If $\wedge$ denotes the maximal configuration in $\O_{L,h}$, i.e.\
$\wedge_i\equiv L+h$, we sometimes consider the model with
$\xi^2=\wedge$ and $\xi^1=\bar\eta$ where $\bar \eta_i := \lfloor i h/(L+1)\rfloor$ for all $i=1,\dots,L$, 
i.e.\ the interface is above the straight line connecting the two
boundary values, cf. Figure \ref{fig:sos} (b). In this case one speaks simply of an {\em interface above the wall}. 
%By symmetry, if $\vee$ stands for the minimal element of $\O_{L,h}$,
%$\vee_i\equiv -L$, then the model with $\xi^1=\vee$ and
%$\xi^2=\bar\eta$ describes an interface below the wall. 
Note that $\bar \eta$ is the SOS equivalent of the ``monotone surface
$\bar\phi^{\bf n}$
with fixed slope'', cf. Definition \ref{def:nu}.

In the following, whenever we do not explicitly mention floor and ceiling,
it is understood that we are talking about the bounded model where
$\xi^1=\wedge$ and $\xi^2=\vee$, where $\vee$ is the minimal
configuration in $\O_{L,h}$: $\vee_i\equiv -L$.

\subsection{Dynamics}
The evolution of the interface is given by the standard heat bath dynamics, i.e.\ single-site Glauber dynamics described as follows. 
There are independent Poisson clocks with mean $1$ at each site $i\in\{1,\dots,L\}$. When site $i$ rings, the height $\eta_i$ is updated to 
the new value $\max\{\eta_i-1,-L\}$ or $\min\{\eta_i+1,L+h\}$  
with probabilities $p_{i,-}(\eta)$, $p_{i,+}(\eta)$ respectively, determined by: 
\begin{gather}\label{ppm}
p_{i,-}(\eta)=\frac{e^{-2} }{1+e^{-2}}\,1_{\{\eta_i \leq a\}} + \frac12\,1_{\{b\geq \eta_i > a\}} 
+ \frac{1}{1+e^{-2}}\,1_{\{\eta_i > b\}}
\nonumber\\
%\frac12\begin{cases}
%e^{-2%\beta
%} & \text{if}\;\eta_i \leq a\\
%1 & \text{otherwise}
%\end{cases}
%\;\;\quad \quad 
p_{i,+}(\eta)=\frac{e^{-2}}{1+e^{-2}}\,1_{\{\eta_i \geq b\}} + \frac12\,1_{\{b> \eta_i \geq a\}} 
+ \frac{1}{1+e^{-2}}\,1_{\{\eta_i < a\}}
%\frac12\begin{cases}
%e^{-2%\beta
%} & \text{if}\;\eta_i \geq b\\
%1 & \text{otherwise}
%\end{cases}
\end{gather}
where $a:=\min\{\eta_{i-1},\eta_{i+1}\}$ 
and $b:=\max\{\eta_{i-1},\eta_{i+1}\}$. With the remaining
probability $1 - (p_{i,-}(\eta) + p_{i,+}(\eta))$, $\eta_i$ stays at its current value. 
%
%if $\eta_i \leq a$
%then $p_{i,-} = \frac{e^{-2}}{2}$, else $p_{i,-} = \frac12$; if $\eta_i\geq b$ then $p_{i,+} = \frac{e^{-2}}2$, 
%else $p_{i,+} = \frac12$. (With the remaining
%probability $1 - (p_{i,-} + p_{i,+})$, do nothing). 
It is not hard to check that this defines a 
continuous time Markov chain with state space $\O_{L,h}$ and
stationary reversible measure given by $\pi_{L,h}$. %Moreover the chain
%is monotone in the sense discussed in Section \ref{}. 
In the sequel
we will write $\eta^\xi(t)$ for the random variable describing the
state of the Markov chain at time $t$ with initial state $\xi$ and $\mu_t^\xi$ for its distribution. 
Let $\tmix=\tmix(L,h)$ denote the mixing time of this Markov chain. 

We may consider the evolution of the system under hard wall constraints $\xi^1,\xi^2$ as above. 
This amounts to the same %consists in the same 
dynamics except that any update which
would violate the constraints %$\eta_i \geq ih/L$ 
is rejected. The dynamics is then reversible w.r.t.\ the equilibrium measure $\pi_{L,h}^{\xi^1,\xi^2}$
associated to the floor $\xi^1$ and the ceiling $\xi^2$, see
also Section \ref{sec:flocei} above
for the analogous constrained dynamics in the monotone surface case. 
In either case, with or without hard walls, the monotonicity considerations recalled in Section \ref{sec:monoton} apply here without modifications, with the natural partial order on configurations introduced above.
% given by 
%$\eta\leq \xi$ iff $\eta_i\leq \xi_i$ for all $i$.

Our main result about the mixing time of the SOS interface is:
\begin{theo}
\label{th:sos}
There exists $\alpha>0$ such that for any $L$  sufficiently large, uniformly in $0\leq h\leq L$ one has
\begin{equation}
  \label{eq:2sos}
  \tmix\le L^2(\log L)^{\alpha}.
  \end{equation}
The same bound holds for the interface constrained to stay  above the wall $\bar\eta$.
\end{theo}
\begin{rem}
  \begin{enumerate}[1.]\ 
  \item It will be clear from the proof that the above bound is satisfied as soon as $\alpha>21$. 
  With some effort this power of $\log L$ can be considerably improved. However, the method we use does not seem to be capable of reaching  the presumably optimal bound $\tmix = O(L^2\log L)$. 
\item % Consider the 
% %It might be useful to consider a slightly 
% more general model %than the one introduced above. Namely, instead of 
% %imposing a homogeneous bottom/top constraints $-L\leq\eta_i\leq L+h$ one can consider
% %general 
% obtained by imposing (inhomogeneous) hard wall  constraints of the form 
% $\xi^1_i\leq\eta_i\leq \xi^2_i$, for arbitrary 
% $\xi^1,\xi^2\in\bbZ$ provided that $h\leq L$ and, for some constant $C>0$ one has  
% $-CL\leq \xi^1_i\leq \xi^2_i\leq C(L+h)$, for $i=1,\dots,L$. 
%It will be clear from the proof that the 
Minor modifications of the proof show that the result of Theorem \ref{th:sos} 
continues to hold as it is for the dynamics constrained between a
floor $\xi^1$ and  a ceiling $\xi^2$ for any $\xi^1,\xi^2\in \Omega_{L,h}$ such that $\xi^1\leq \bar\eta\leq \xi^2$.
\item A further extension is obtained by letting the  
%bottom/top 
constraints $-L\leq\eta_i\leq L+h$ be replaced by $-M\leq\eta_i\leq M+h$, where $M$ 
is an independent parameter, possibly much larger than $L$. %As discussed in Remark \ref{} below, 
It is possible to extend the proof of Theorem \ref{th:sos} to get the essentially sharp 
bound $\tmix = \tilde O(L\max(L,M))$.  \end{enumerate}
\end{rem}

%add remarks why bounded/unbounded

\subsection{Spectral gap}\label{sec:gap}
The unbounded version of the SOS model is given in \eqref{eq:unbo}.
We write again $\wt\pi=\wt\pi_{L,h}$ for the corresponding Gibbs
measure.  The dynamics is the same as above except for the absence of
the constraints $\vee\le \eta\le \wedge$, i.e.\ when the clock
labeled $i$ rings, $\eta_i$ is updated to the new value $\eta_i\pm 1$ with probability $p_{i,\pm}$ given by \eqref{ppm}. 
The infinitesimal generator is given by
\begin{equation}\label{gener}
\cL f(\eta) = \sum_{i=1}^L \big
\{p_{i,+}(\eta) \nabla_{i,+}f(\eta) +  p_{i,-}(\eta) \nabla_{i,-}f(\eta)
\big\}
\,,
\end{equation}
where $\nabla_{i,\pm}f(\eta) = f(\eta^{i,\pm})-f(\eta)$, and $\eta^{i,\pm}$ is the configuration coinciding with $\eta$ everywhere except that at site $i$ the value of $\eta_i$ is replaced by $\eta_i\pm 1$.
The Dirichlet form is given by
\begin{equation}\label{dirich}
\cE(f,f) = \frac12\sum_{i=1}^L \wt\pi\big[
p_{i,+} (\nabla_{i,+}f)^2 +  p_{i,-} (\nabla_{i,-}f)^2
\big]
\end{equation}
where $\wt\pi[\cdot]$ denotes expectation w.r.t.\ $\wt\pi$. Note that, for each finite $L$, $\cL$ defines a bounded self-adjoint operator in $L^2(\bbZ^L,\wt\pi)$. 
The associated spectral gap is  defined by \eqref{defgap}, where $f$ ranges over all $f\in L^2(\bbZ^L,\wt\pi)$ with nonzero variance.
\begin{theo}
\label{th:gap}
Let $\alpha>0$ be as in Theorem \ref{th:sos}. 
For some constant $c>0$, for all $0\leq h\leq L$, the spectral gap of the unbounded SOS dynamics satisfies 
\begin{equation}
  \label{eq:gap}
  \gap(L)\ge c\,L^{-2}(\log L)^{-\alpha}.
  \end{equation}
The same bound holds for the interface constrained to stay above the
wall $\bar \eta$. 
\end{theo}
This estimate is optimal, modulo the logarithmic factor: an upper
bound $O(1/L^2)$ is given e.g. in \cite{MS}. 
The lower bound $\gap(L)\ge c/L^2$  was proven in \cite{Posta}
 for a modified
version of the SOS model with weak boundary couplings;  such modified model is much less sensitive to the
boundary conditions and has a genuinely different dynamical behavior. 
If instead the absolute value
interaction potential were replaced by one with strictly convex
behavior at infinity, then
the correct lower bound $\gap(L)\ge c/L^2$ would follow by well-established
recursive methods, see e.g. \cite{C}.

\section{Strategy of the proof}
\label{sec:scale}
As already announced in the introduction, despite the fact that the
equilibrium fluctuations of the interface in the two models are very
different, our bound $\tmix= \wt O(L^2)$ is proved following a
common strategy that we sketch here.

The crucial step is the following (see Propositions \ref{prop:maxmin} and \ref{prop1} below for a precise
formulation in the case of monotone surfaces and SOS model):
\begin{step}%[Getting to the right scale]
\label{claim:1}
  Starting from any
configuration, after time $\wt O(L^2)$ the distance between the
interface and the flat profile is not larger than its typical
equilibrium value 
$\chi_L$. 
\end{step}
At that point one can conclude $\tmix=\wt O(L^2)$ provided that a
result of the following type is available (see e.g. Proposition \ref{th:key} below in the case of 
the SOS model):
\begin{step}\label{th:vago1}
  If the initial condition $\xi$ is at distance $\chi_L$ from the flat
  profile, then $\|\mu^\xi_T-\pi\|\ll 1$ for some $T=\wt O(L^2)$.
\end{step}

The proof of such result is model-dependent: for the SOS model it
was given in \cite{MS} and for monotone surfaces it follows
from results in \cite{CMST} (see Propositions \ref{prop:quasisoff} and
\ref{prop:wilson} below) .

In turn, Step \ref{claim:1} follows if one proves that, with high probability, the
interface started from the maximal  configuration stays
below a \emph{deterministic interface evolution}  which
after time $\wt O(L^2)$ is at the correct distance $O(\chi_L)$ from the flat
profile. 
It turns out that it is actually sufficient to define the
deterministic interface evolution along a sequence of deterministic times
$t_n,0\le n\le M$. At all
times $t_n$, the deterministic interface is the boundary of $\cC_{u_n}$ where, given
$u>0$, $\mathcal C_{u}$ is a spherical cap (if we are considering two-dimensional
interfaces like monotone surfaces) or a circular segment (in the case of
one-dimensional
interfaces like SOS) of height $u$ and base of linear size $\rho_L$
roughly of order $ L$, see Figure \ref{fig:cun}. The base of $\cC_u$ lies on the plane/line which contains the macroscopic flat profile.
The evolution of $\cC_{u}$, by a kind of ``flattening
process'', in the time interval $(t_{n-1},t_n]$ transforms $\mathcal
C_{u_{n-1}}$ into $\mathcal C_{u_n}$. The sequence  of
increasing times $\{t_n\}_{0\le n\le M}$ and
of 
decreasing heights $\{u_n\}_{0\le n\le M}$   will be introduced in a moment. 
The ``domination statement'' then is of the following type (see Propositions \ref{lemma:duro} and \ref{propC}):
\begin{claim}\label{th:vago2}
  For all $0\le n\le M$, with high probability the following
  holds. For all times in $[t_n,L^3]$ the evolution started from the
  maximal configuration stays below the boundary of $\cC_{u_n}$. 
\end{claim}

The initial height $u_0$ is taken to be proportional to $L$ and one
sets $t_0=0$; this guarantees that the statement of Claim
\ref{th:vago2} holds trivially for $n=0$. In order
to choose $u_{n+1}$ given $u_n$, one uses  the following procedure.
Consider the spherical cap/circular segment $\cC_{u_n}$ and choose a point on its curved
boundary (e.g. the highest one). Move inward (i.e. inside $\cC_{u_n}$)
the tangent plane/line
at the chosen point by an amount $\Delta$ and call $d_\Delta$ the diameter of the intersection
between the plane/line with $\cC_{u_n}$, see Figure \ref{fig:cun}. 
\bigskip

\begin{figure}[h]
\centerline{
\includegraphics[scale=2, width=16cm]{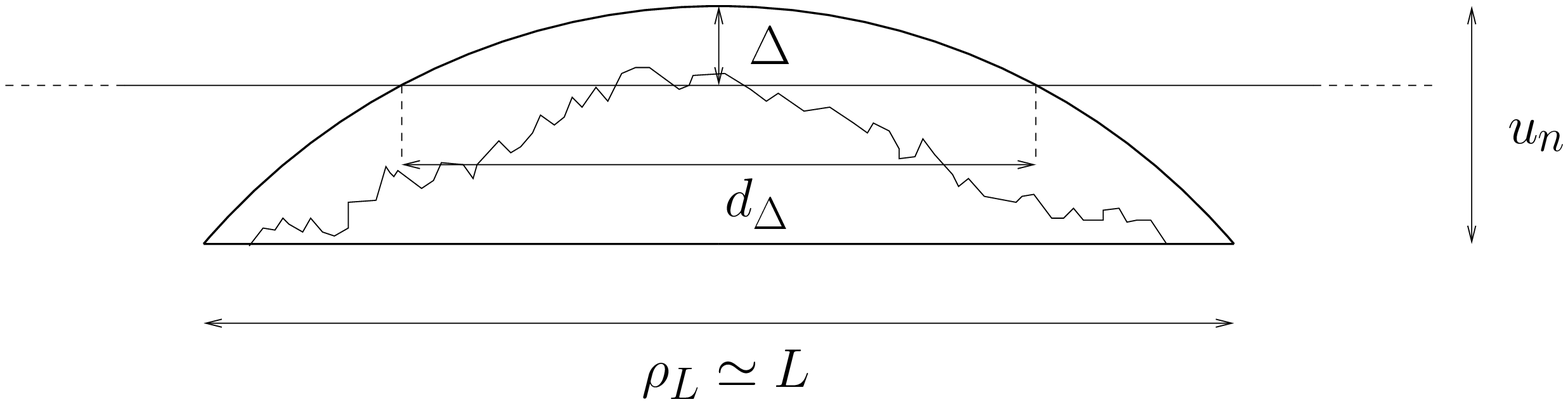}}
\caption{The spherical cap/circular segment $\cC_{u_n}$. The
  deterministic evolution at time $t_n$ coincides with the curved
  portion of the boundary of  $\cC_{u_n}$. At that time, the true stochastic evolution
  (wiggled line) stays with high probability below it. Elementary
  geometry shows that $\Delta\simeq d_\Delta^2(u_n/\rho_L^2)$. The requirement
  $u_n-u_{n+1}=\Delta_n\simeq d_{\Delta_n}^\gamma$ then leads to 
$u_n- u_{n+1}\simeq (\rho_L^2/u_n)^{\gamma/(2-\gamma)}$.}
\label{fig:cun}
\end{figure}
Then $u_{n}-u_{n+1}$ is
chosen as 
the
critical value $\Delta_n$ such that the equilibrium fluctuations on scale
$d_{\Delta_n}$ are of order $\Delta_n$ (apart from logarithmic
corrections), i.e. $\Delta_n=\wt O(\chi_{d_{\Delta_n}})$. Also, $M$
is the smallest index such that $u_M\le \chi_L$,
i.e. $u_M$
is of the order of the equilibrium
height fluctuations on scale $L$.
As for the
time sequence $\{t_n\}_{0\le n\le M}$, one sets
$t_{n+1}-t_n$ to be of order $d_{\Delta_n}^2$ (again
neglecting logarithmic corrections): that this is the correct
choice is guaranteed by a careful use of Step \ref{th:vago1}, applied
with $L=d_{\Delta_n}$. 
It is not difficult to realize that $t_M=\wt O(L^2)$.
Indeed, assume for definiteness that $\chi_L\sim L^\gamma$ for some $0\le \gamma<1$, where
if $\gamma=0$ we mean that $\chi_L\sim \text{polylog}(L)$. Then,
simple geometric considerations show that 
\[
u_n-u_{n+1}=\wt O \left(\left(\frac{\rho_L^2}{u_n}\right) ^{\gamma/(2-\gamma)}
\right),\;\;\;t_{n+1}-t_n=\wt O\left(\left(\frac{\rho_L^2}{u_n}\right) ^{2/(2-\gamma)}
\right).
\]
Approximating the recursion for $u_n$ with a differential equation gives 
\[
u_n^{2/(2-\gamma)}\simeq
u_0^{2/(2-\gamma)}-\rho_L^{2\gamma/(2-\gamma)}n
\simeq L^{2\gamma/(2-\gamma)}\left[L^{(2-2\gamma)/(2-\gamma)}-n\right]
\]
since both $u_0$ and $\rho_L$ are of order $L$.
In particular, one has roughly $M= O(L^{(2-2\gamma)/(2-\gamma)})$.
Then, 
\[
t_M=\sum_{n=0}^{M-1}(t_{n+1}-t_n)\simeq
\rho_L^{4/(2-\gamma)}\sum_{n=0}^{M-1}\frac1{u_n^{2/(2-\gamma)}}\simeq
\frac{\rho_L^{4/(2-\gamma)}}{
  L^{2\gamma/(2-\gamma)}}\sum_{n=0}^{M-1}\frac1{L^{(2-2\gamma)/(2-\gamma)}-n}
=\wt O(L^2)
\]
since the last sum is of order $\log L$.
Remarkably, the order of magnitude of $t_M$ does not depend on the
fluctuation exponent $\gamma$, while the sequence $(t_n,u_n)$ and the
value of $M$ do.
The statement of Claim \ref{th:vago2} for $n=M$ allows to conclude Step \ref{claim:1}: the
evolution started from the maximal configuration,  at time $t_M=\wt
O(L^2)$,  is below the
deterministic evolution, which is within distance
$\chi_L$ from the flat profile.

Another way to understand the choice of the time-scales $t_n$ is the
following. If one imagines that the boundary of  $\cC_u$  evolves by ``mean
curvature'', i.e.  feeling a inward drift proportional to the inverse of
its instantaneous radius of curvature, then the time $t_{n+1}-t_n$ to transform $\mathcal
C_{u_{n}}$ into $\mathcal C_{u_{n+1}}$ must be $O(R_n\times
(u_{n}-u_{n+1}))$, where $R_n$ is the radius of curvature of
$\cC_{u_n}$. 
One can easily check that, apart from logarithmic
corrections, this coincides with the requirement $t_{n+1}-t_n \simeq d_{\Delta_n}^2$.

\section{Monotone surfaces: Proof of Theorem \ref{th:eq3d}}

An essential tool in the proof of Theorem \ref{th:eq3d} are the \emph{ translation
  invariant, ergodic Gibbs measures}, with given slope ${\bf n}$, on the set $\Omega$ of monotone surfaces
\cite{KOS,Sheffield}. The trick of using the properties of such infinite-volume states to obtain fluctuation bounds for surfaces with fixed boundary conditions
around  a finite region was already crucial in \cite{CMST}. 
The most relevant result for this purpose is the the following.
\begin{theo}\cite{KOS,Sheffield}
Given ${\bf n}>0$, $h\in
\ZZ$ and $\bar x\in \ZZ^2$, there exists a unique
law $\mu_{{\bf n},h,\bar x}$ on $\Omega$ such that 
\begin{enumerate}[(i)]
\item $\mu_{{\bf n},h,\bar x}(\phi_{\bar x}=h)=1$;
\item for every $x\in \ZZ^2$, 
$\mu_{{\bf n},h,\bar x}(\phi_x)=h+
  \Pi^{\bf n}(x-\bar x)$ (cf. Definition \ref{def:nu});

\item for every $x_1,\ldots,x_k\in \ZZ^2$ and $v\in\ZZ^2$ one has that the joint law of $\{\phi_{x_i}-\phi_{x_j}\}_{1\le i,j\le k}$ is the same as the joint law of 
$\{\phi_{x_i+v}-\phi_{x_j+v}\}_{1\le i,j\le k}$ (translation invariance of the law of the height gradients);
\item for every finite $U\subset \ZZ^2$ such that $\bar x\notin U$ and every monotone surface
  $\xi\in\Omega$ such that $\xi_{\bar x}=h$, the measure $\mu_{{\bf
      n},h,\bar x}$ conditioned to
  $\{\phi_x=\xi_x$ for every $x\notin U\}$ is $\pi^{\xi}_U$, i.e. the uniform measure over
  all $\phi\in\Omega$ which coincide with $\xi$ outside $U$ (DLR property).

\end{enumerate}
\end{theo}
Concerning height fluctuations under these Gibbs measures one can
prove that for every $\epsilon>0$ 
there exists a positive constant $c$ such that, for every  $a>0$ and
$L$ large enough, 
  \begin{eqnarray}
    \label{eq:17bis}
    \mu_{{\bf n},h,\bar x}\left(\exists y\in \ZZ^2:%\mbox{\;s.t. \;}
    \;d(\bar
      x,y)\le
    L\mbox{\;and\;}
|\phi_y-\mu_{{\bf n},h,\bar x}(\phi_y)|\ge a(\log L)^{1+\epsilon}\right)\le\frac1c e^{-c \,a (\log L)^{1+\epsilon}},
  \end{eqnarray}
(see \cite[Proposition 5.7]{CMST};
the proof is given there for $\epsilon=1/2$ and $a=1/8$ but it works identically for
$\epsilon>0,a>0$).
In other words, the surface lies on average on a plane of slope ${\bf n}$
and the height difference between two points $x,y$ does not differ from the
average height difference by much more than the logarithm of $d(x,y)$.
 The estimate \eqref{eq:17bis}  is obtained via the well-known fact
\cite{Kenyon-lecturenotes} that height differences between two points
can be written as the number of points of a determinantal point
process whose kernel is explicitly known.

The connection between the measure $ \mu_{{\bf n},h,\bar x}$ and  the ``infinite volume states'' for dimer coverings
of the infinite honeycomb lattice defined in \cite{KOS} is well known
and is discussed
for instance in \cite[Section 5]{CMST}. Let us just recall that the
components ${\bf n}^{(a)}, a=1,2,3$ are directly related to the fractions of
dimers of the three possible  types (horizontal or rotated by $\pi/3$ or by $(2/3)\pi$).

We now turn to the proof of \eqref{eq:17}. By symmetry it is enough to
show that 
 \begin{eqnarray}
    \label{eq:17tris}
    \pi^{\eta}_U\left(\exists y\in U:%\mbox{\;s.t. \;}
\phi_y \ge \bar \phi^{\bf n}_y + a(\log L)^{1+\epsilon}\right)\le(1/c) e^{-c\,a(\log L)^{1+\epsilon}}.
  \end{eqnarray}
By monotonicity the probability in the left-hand side of \eqref{eq:17tris} is increased whenever $\eta$ is replaced by 
$\eta'$ such that 
%  \begin{eqnarray}   
%     \pi^{\eta_{\partial U}}_U\left(\exists y\in U:%\mbox{\;s.t. \;}
% \phi_y\ge\bar \phi_y+\frac18(\log L)^{1+\epsilon}\right)\le
%     \pi^{\eta'_{\partial U}}_U\left(\exists y\in U:%\mbox{\;s.t. \;}
% \phi_y\ge\bar \phi_y+\frac18(\log L)^{1+\epsilon}\right)
%   \end{eqnarray}
$\eta_{\partial U}\le \eta'_{\partial U}$.
In particular, let $\eta'$ be sampled from $\mu_{{\bf n},h,\bar x}(\cdot|A)$ where $\bar x$ is a given element of $\partial U$, $h$ is such that
$\mu_{{\bf n},h,\bar x}(\phi_{\bar x})=\eta_{\bar x}+(a/2)(\log L)^{1+\epsilon} $
and $A$ is the event $A=A_{\eta,U}=\{\phi\in\Omega:\phi_{\partial U}\ge \eta_{\partial U}\}$.
From \eqref{eq:17bis} and the fact that $diam(U)=L$ one sees that
\[
\|\mu_{{\bf n},h,\bar x}(\cdot|A)-\mu_{{\bf n},h,\bar x}\|=O\left[\exp\left(-\frac a3\,c\,(\log L)^{1+\epsilon}\right)\right].
\]
This is because, if $y\in\partial U$, $\phi_y<\eta_y$ implies $\phi_y< \mu_{{\bf n},h,\bar x}(\phi_y)-(a/2+o(1))(\log L)^{1+\epsilon}$, 
recall Definition \ref{ass:eta}, so that $\mu_{{\bf n},h,\bar x}(A^c)=O(\exp(-(a/3)\,c\,(\log L)^{1+\epsilon}))$.
Therefore, the  probability in the left-hand side of \eqref{eq:17tris} is upper bounded by
\begin{gather*}
 \int\mu_{{\bf n},h,\bar x}(d\eta')\,\pi^{\eta'}_U\left(\exists y\in U:%\mbox{\;s.t. \;}
\phi_y-\bar \phi^{\bf n}_y\ge a(\log L)^{1+\epsilon}\right)
+O\left[\exp\left(-\frac a3\,c\,(\log L)^{1+\epsilon}\right)\right]\\
=
 \mu_{{\bf n},h,\bar x}\left(\exists y\in U:%\mbox{\;s.t. \;}
\phi_y-\bar \phi^{\bf n}_y\ge a(\log L)^{1+\epsilon}\right)
+O\left[\exp\left(-\frac a3\,c\,(\log L)^{1+\epsilon}\right)\right]  \\
=O\left[\exp\left(-\frac a3\,c\,(\log L)^{1+\epsilon}\right)\right]
\end{gather*}
where in the first step we used  the DLR property and in the second one the fluctuation bound \eqref{eq:17bis}, together with 
the fact that with our choice of $h$ one has $\mu_{{\bf n},h,\bar x}(\phi_y)=\bar \phi^{\bf n}_y+(a/2+o(1))(\log L)^{1+\epsilon}$.

\section{Monotone surfaces: Proof of Theorem \ref{th:3d}}
\label{3d}
In this section the slope ${\bf n}>0$, the region $U$ and the
good planar boundary condition $\eta$ with slope $\bf n$ are fixed as in Theorems
\ref{th:eq3d}, \ref{th:3d}.  Denote $\wedge$ (resp. $\vee$) the maximal
(resp. minimal) configuration of $\Omega_{\eta,U}$ with respect to
the partial ordering ``$\le$'' and recall that $\phi^\xi(t)$ denotes the
configuration at time $t$ started from $\xi$.

The first key ingredient is a result saying that, after time of order  $L^2(\log
L)^{12}$, the surface is at most at distance $(\log L)^{3/2}$ away from
the plane $\Pi^{\bf n}$ of slope ${\bf n}$ (cf. Definition \ref{def:nu}). It is here that the new ideas of mimicking the evolution by mean curvature play a crucial role. 
\begin{prop}
\label{prop:maxmin}
Let $T=(1/2)L^2(\log
L)^{12}$. Then, there exists $c>0$ such that
\begin{eqnarray}
  \label{eq:5}
  \bbP\left(\max_{x\in U} (\phi_x^\wedge(T)-\bar \phi^{\bf n}_x))> (\log
    L)^{3/2}\right)=O\left(e^{-c(\log L)^{3/2}}\right)
\end{eqnarray}
  and similarly 
\begin{eqnarray}
  \label{eq:5bis}
  \bbP\left(\min_{x\in U} (\phi_x^\vee(T)-\bar \phi^{\bf n}_x)<- (\log
    L)^{3/2}\right) =O\left(e^{-c(\log L)^{3/2}}\right).
\end{eqnarray}
\end{prop}
The second result says that once the surface is within distance $(\log
L)^{3/2}$ from $\Pi^{\bf n}$, it does not go much farther than that for a time
much longer than $L^{2}(\log L)^{12}$. This second step is much more
standard and its proof combines monotonicity and reversibility together with the fluctuation bounds of Theorem \ref{th:eq3d}. 
\begin{definition}
  \label{def:p=-}
Let $\Phi^+$ (resp. $\Phi^-$) be the maximal (resp. minimal)
configuration in the 
set
\[
\{\phi\in\Omega_{\eta,U}: \max_{x\in U} |\phi_x-\bar \phi^{\bf n}_x|\le 2(\log
    L)^{3/2}\}.
\]
\end{definition}

\begin{prop}
\label{prop:quasisoff}
  Let $\xi\in \Omega_{\eta,U}$ be such that
  \begin{eqnarray}
    \label{eq:7}
   \max_{x\in U} |\xi_x -\bar \phi^{\bf n}_x|\le (\log
L)^{3/2}.
  \end{eqnarray}
Then, there exists $c>0$ such that
\begin{eqnarray}
  \label{eq:8}
  \bbP\left(
\exists t<L^{10}:\phi^\xi(t)\notin \Omega_{\eta,U}^{\Phi^\pm}
\right) = O\left(e^{-c(\log L)^{3/2}}\right)
\end{eqnarray}
% and 
% \begin{eqnarray}
%   \label{eq:15}
%   \pi(\Omega_{\eta,U}^{ \Phi^\pm})\ge 1-O\left(e^{-c(\log L)^{3/2}}\right),
% \end{eqnarray}
with $\Omega_{\eta,U}^{ \Phi^\pm}$ given in \eqref{eq:9} and $\Phi^\pm$ as in Definition \ref{def:p=-}.
\end{prop}

Finally, the last step  shows that if the surface evolves constrained
between a ceiling and a floor which are within distance $O((\log
L)^{3/2})$ from $\Pi^{\bf n}$, then mixing occurs within a time $\tilde O(L^2)$.
\begin{prop}\cite[Theorem 4.3]{CMST}
  \label{prop:wilson} Let $\phi^\pm\in \Omega_{\eta,U}$ with $\phi^-\le\phi^+$.
For the dynamics restricted to $\Omega_{\eta,U}^{ \phi^\pm}$
(cf. Section \ref{sec:flocei}) one has \[\tmix=O(L^2(\log L)^2 H^2)\]
where $H=\max_{x\in U}(\phi^+_x-\phi^-_x)$.
\end{prop}

\medskip
We can now easily put together Propositions
\ref{prop:maxmin} to \ref{prop:wilson} to obtain the desired upper bound \eqref{eq:2}
on the mixing time:
\begin{proof}[Proof of Theorem \ref{th:3d}]
It is a standard fact that 
  \begin{eqnarray}
    \label{eq:11}
    \max_{\xi\in \Omega_{\eta,U}}\|\mu^\xi_t-\pi\|\le L^3 \max(\|\mu^\wedge_t-\pi\| ,\|\mu^\vee_t-\pi\| )
  \end{eqnarray}
(see e.g. \cite[Lemma 6.2]{CMST} for a similar statement) so it is
sufficient to prove that 
\[\max(\|\mu^\wedge_t-\pi\|
,\|\mu^\vee_t-\pi\|)\le 1/(2e L^3)\] for $t=L^2(\log L)^{12}=2 T$. Let us
consider e.g. the case of the maximal initial condition $\wedge$, the
other case being analogous. 

Define $\Omega'=\{\phi\in\Omega_{\eta,U}: \max_{x\in U} |\phi_x-\bar \phi^{\bf n}_x|\le (\log
    L)^{3/2}\}$ and, for $\xi\in \Omega'$,
\[
\tau=\inf\{t>0: \max_{x\in U}|\phi^\xi_x(t)-\bar \phi^{\bf n}_x
|\ge 2(\log L)^{3/2}-1\}.
\]
 Let $A$ be a subset of $\Omega_{\eta,U}$. 
Then, using Proposition \ref{prop:maxmin},
\begin{gather*}
  \label{eq:13}
  \mu^\wedge_{2T}(A)=\mu^\wedge_T(\mu^\xi_T(A)|\xi\in
  \Omega')+O\left(e^{-c(\log L)^{3/2}}\right).
\end{gather*}
 Next, from Proposition
\ref{prop:quasisoff}, one has for every $\xi\in\Omega'$
\begin{gather*}
  \label{eq:14}
  \mu^\xi_T(A)=\bbP(\phi^\xi(T)\in A;
  \tau>T)+O\left(e^{-c(\log L)^{3/2}}\right) =\bbP^{\Phi^\pm}(\phi^\xi(T)\in A;
  \tau>T)+O\left(e^{-c(\log L)^{3/2}}\right) \\=
\bbP^{\Phi^\pm}(\phi^\xi(T)\in A)+O\left(e^{-c(\log L)^{3/2}}\right) 
\end{gather*}
where $\bbP^{\Phi^\pm}$ denotes the law of the dynamics restricted to
the set $\Omega^{\Phi^\pm}_{\eta,U}$. Indeed, up to the random time $\tau$
the two dynamics $\bbP$ and $\bbP^{\Phi^\pm}$ can be perfectly coupled so that they coincide. In particular, 
$\tau$ has the same
law under $\bbP$ and $\bbP^{\Phi^\pm}$. Finally, thanks to 
Proposition \ref{prop:wilson}, $T$ is at least $(\log L)^2$ times
the mixing time of the restricted
dynamics (which is $O(L^2(\log L)^5)$). Therefore, from 
  \eqref{eq:10} and the fact that the invariant measure of the restricted dynamics is $\pi(\cdot|\Omega^{\Phi^\pm}_{\eta,U})$, one has 
  \begin{eqnarray}
    \label{eq:16}
    |\bbP^{\Phi^\pm}(\phi^\xi(T)\in A)-\pi(A|\Omega^{\Phi^\pm}_{\eta,U})|\le e^{-(\log L)^2}.
  \end{eqnarray}
Thanks to Theorem \ref{th:eq3d}  one has $\pi(\Omega^{\Phi^\pm}_{\eta,U})\ge 1-O(\exp(-c(\log L)^{3/2}))$ and finally
\[
|\mu^\wedge_{2T}(A)-\pi(A)|=O\left(e^{-c(\log L)^{3/2}}\right) 
\]
for every event $A\in\Omega_{\eta,U}$, which implies $\|\mu^\wedge_{2T}-\pi\|\le 1/(2e
L^3)$ for $L$ large enough.
\end{proof}

\medskip
\smallskip
As a warm-up, we start by proving the easier Proposition  \ref{prop:quasisoff}.
   \begin{proof}[Proof of Proposition \ref{prop:quasisoff}] Let
     $\xi\in \O_{\eta,U}$
    satisfy \eqref{eq:7}.
By monotonicity, if $\eta'\in\Omega$ and $\xi'\in \O_{\eta',U}$ are such that $\xi\le \xi'$ and $\eta\le \eta'$, then   
\begin{gather}
\label{eq:monoton}
 \bbP\left(
\exists t<L^{10}:\max_{x\in U}\left(\phi^\xi_x(t)-\bar\phi^{\bf n}_x\right)\ge 2(\log L)^{3/2}
\right) \\
\le
 \bbP'\left(
\exists t<L^{10}:\max_{x\in U}\left(\phi^{\xi'}_x(t)-\bar\phi^{\bf n}_x\right)\ge 2(\log L)^{3/2}
\right) 
  \end{gather}
where $\bbP'$ denotes the evolution with boundary condition
$\eta'$
(instead of $\eta$). In particular, this is the case if we set $\eta'_x=\eta_x+\lfloor (3/2)(\log L)^{3/2}\rfloor$
and $\xi'$
is sampled from the measure $\pi^{\eta'}_U(\cdot|A)$,
where $A$ is the event $A=A_{\eta',\xi,U}=
\{\phi\in \Omega_{\eta',U}:\phi_U\ge \xi_U\}$.
From Theorem \ref{th:eq3d} (applied with $\epsilon=a=1/2$) one sees that
 \begin{gather}
   \label{eq:20}
\|\pi^{\eta'}_U(\cdot|A)-\pi^{\eta'}_U\|=O\left(e^{-c(\log L)^{3/2}}\right)
 \end{gather}
for some $c>0$.
This is because $\eta'$ is within distance $C\log L$ from the plane $\Pi^{\bf n}_{(3/2)(\log L)^{3/2}}$ (cf. Definition \ref{def:nu})
while $\xi$ is within distance $(\log L)^{3/2}$ from   the plane $\Pi^{\bf n}$.
% This is because, if e.g. $y\in \partial U$, then (recall condition
% \eqref{eq:4})
% \[
% \mu_{{\bf n},h,\bar x}[\phi(y)]\ge \eta(y)+(3/2+o(1))(\log L)^{3/2}
% \]
% so that 
% \[
% \mu_{{\bf n},h,\bar x}[\phi(y)<\eta(y)]\le \mu_{{\bf n},h,\bar
%   x}[|\phi(y)- \mu_{{\bf n},h,\bar
%   x}(\xi'(y))|>
% (3/2+o(1))(\log L)^{3/2}
% ]=O\left(e^{-c(\log L)^{3/2}}\right).
% \]
% A similar argument works to bound $\mu_{{\bf n},h,\bar
%   x}(\phi(y)<\xi(y))$, $y\in U$ and then a union bound over $y\in
% U\cup \partial U$ implies \eqref{eq:20}.
Therefore, the probability in \eqref{eq:monoton} is upper bounded by 
\begin{eqnarray}
  \label{eq:21}
  \int \pi^{\eta'}_U(d\xi')\ \bbP'\left(
\exists t<L^{10}:\max_{x\in U}\left(\phi^{\xi'}_x(t)-\bar\phi^{\bf n}_x\right)\ge 2(\log L)^{3/2}
\right) +O\left(e^{-c(\log L)^{3/2}}\right).
\end{eqnarray}
The initial condition $\xi'$ in \eqref{eq:21} is sampled from $ \pi^{\eta'}_U$, which is the invariant measure of the dynamics 
$\bbP'$,
so that the distribution of $\phi^{\xi'}(t)$ coincides with $ \pi^{\eta'}_U$ at all later times.
Via a union bound over times and recalling the relation between $\eta$ and $\eta'$, the first term in \eqref{eq:21} is upper bounded by
\begin{gather}
\label{eq:unbou}
 \pi^{\eta'}_U\left(
\max_{x\in U}\left(\phi_x-\bar\phi^{\bf n}_x\right)\ge 2(\log L)^{3/2}
\right) \times
O(L^{10}\times L^2)\\
\le  \pi^{\eta}_U\left(
\max_{x\in U}\left(\phi_x-\bar\phi^{\bf n}_x\right)\ge (1/2)(\log L)^{3/2}-1
\right) \times
O(L^{10}\times L^2)
  \end{gather}
which is of order $\exp(-c(\log L)^{3/2})$, see Theorem \ref{th:eq3d}. %Note in
%fact that the height $\bar \phi_x$ is at distance 
%$(3/2+o(1))(\log L)^{3/2}$ away from the plane $\Pi_{(3/2)(\log L)^{3/2}}$.
The
factor $O(L^{10}\times L^2)$ is just the average number of Markov chain moves within time
$L^{10}$, since there are order of $L^2$ lattice points in $U$.
Similarly one bounds the probability that
$\min_{x\in U}(\phi^\xi_x(t)-\bar\phi^{\bf n}_x)<-2(\log L)^{3/2}$
for some $t\le L^{10}$ and claim \eqref{eq:8} is proven.
 \end{proof}

  \begin{proof}[Proof of Proposition \ref{prop:maxmin}]
We prove only Eq. \eqref{eq:5} since \eqref{eq:5bis}  is obtained
essentially in the same way.
    Let $W$ be a disk of radius \[\rho_L=L\times\log L\] on the plane
    $\Pi^{\bf n}_{C\log L}$ of slope ${\bf n}$ (cf. Definition \ref{def:nu}, with $C$ the same
    constant as in \eqref{eq:4}) such that 
its projection $V$ on the horizontal plane contains $U$ and moreover the distance between $\partial U$ and $\partial V$ is at least $\rho_L/2$
(recall that $U$ has diameter $L$). 

Given $u>0$, let  $\mathcal C_u$
be the spherical cap whose base is the disk $W$ and whose height is
$u$. The radius of curvature $R$ is related to $u$ and $\rho_L$ by
\begin{eqnarray}
\label{eq:12}
(2R-u)u=\rho_L^2
\end{eqnarray}
and, since we will always work under the condition $u\ll \rho_L\ll R$, we have 
$R=\rho_L^2/(2u)(1+o(1))$.
For a point $v$ on the curved portion of the boundary of $\mathcal
C_u$, let ${\bf n}_v$ be the normal at $v$ directed towards the
exterior of $\mathcal
C_u$. It is clear that,
if $u\ll\rho_L$, one has ${\bf n}_v={\bf
  n}+o(1)$; in particular, ${\bf n}_v>0$ with the convention of
Definition \ref{def:nu}.
Finally the height (w.r.t. the horizontal plane) of the spherical cap at horizontal coordinates
$x\in V$ is denoted by \[\psi_u(x)=\sup\{z\in \RR:
(x^{(1)},x^{(2)},z)\in \mathcal C_u\}.\]

We now define a sequence of spherical caps $\{\mathcal
C_{u_n}\}_{n=0}^M$ with constant base $W$, decreasing height $u_n$ and
increasing radius of curvature $R_n$. More precisely, let $u_0=2L$ and
$M:=2L-(\log L)^{5/4}$. Then we let
$u_n=u_{n-1}-1=2L-n$ and $(2R_n-u_n)u_n=\rho_L^2$.  
For later purposes we also define 
\[
t_n=t_{n-1}+R_{n}(\log L)^{17/2},\;\;t_0=0.
\]

 \begin{rem}
   \label{rem:odg}
It is worth noting that $R_0\sim L(\log L)^2/4$, $R_M\sim
L^2/(2(\log L)^{3/4})$ and $R_{n+1}/R_n= 1+o(1)$  uniformly 
in the whole
range $n=0,\ldots,M$.
 \end{rem}
 Recalling that $R_n=\rho_L^2/(2 u_n)(1+o(1))$, where $o(1)$ is small
uniformly in $1\le n\le M$, it is immediate to
deduce that
\begin{eqnarray}
  \label{eq:19}
  t_M=\rho_L^2(\log L)^{17/2}\times O\left(\sum_{n=1}^M\frac1{u_n}\right)=O(L^2(\log L)^{23/2}).
\end{eqnarray}
With this notation the key step is represented by the next Proposition.
\begin{prop} \label{lemma:duro}
 There exists a positive constant $c'$ such that the following holds for $L$ large enough.
  For every $0\le n\le M$ one has, with probability at least
$
1-n\,\exp({-c'(\log L)^{3/2}}),
$
  \begin{eqnarray}
    \label{eq:18}
    \phi^\wedge_x(t)\le \psi_{u_n}(x) \mbox{\;for every \;} x\in U
    \mbox{\;and every\;} t\in [t_n,L^3].
  \end{eqnarray}
\end{prop}
Since $u_M\ll (\log L)^{3/2}$, Proposition \ref{lemma:duro} together with
\eqref{eq:19} imply the desired inequality \eqref{eq:5}.
 \end{proof}

 \begin{proof}[Proof of Proposition \ref{lemma:duro}]
   We prove the claim by induction on $n$. For $n=0$ this is trivial
since we chose $u_0=2L$ such as to guarantee that 
the maximal configuration $\wedge\in\Omega_{\eta,U}$ is below the function $U\ni
x\mapsto \psi_{u_0}(x)$.

Assume the claim for some $n$. For $x\in U$ define
the event
\[
A_x=\{\exists t\in [t_{n+1},L^3]:\phi^\wedge_x(t)> \psi_{u_{n+1}}(x)\}
\]
so we need to prove $\bbP(\cup_{x\in U}A_x)\le
(n+1) \exp(-c'(\log L)^{3/2})$. 
We have
\begin{gather}
%\nonumber
 \bbP(\cup_{x\in U}A_x)
\label{eq:gathero}
\le \sum_{x\in U}\bbP(A_x;\ \phi^\wedge(s)\le
  \psi_{u_n}
\mbox{\;for every\;} s\in[t_n,L^3])+n \,e^{-c'(\log L)^{3/2}}
\end{gather}
where we write $\phi^\wedge(s)\le
  \psi_{u_n}$ to mean that $\phi^\wedge_y(s)\le
  \psi_{u_n}(y)$ for every $y\in U$. 

Given $x\in U$, consider the plane $\wt \Pi$ 
tangent to $\mathcal C_{u_n}$ at the point
$(x^{(1)},x^{(2)},\psi_{u_n}(x))$ and the plane $\wt \Pi'$
obtained by translating downwards $\wt \Pi$  by
$(\log L)^{3/2}$. The intersection of $\wt \Pi'$ with
$\mathcal
C_{u_n}$ is a disk $\widetilde W$ of radius $O(\sqrt{R_n}(\log L)^{3/4})$, whose
projection on the horizontal plane we call $Z$. Let $\widetilde U\subset \ZZ^2$
be such that $\wt U\subset  Z$ and $\partial \wt U$ is at
distance of order $1$ from $\partial Z$, so that
of course $diam(\tilde U)=O(\sqrt{R_n}(\log L)^{3/4})$.

Let $\wedge^{(n)}\in\Omega$ be the maximal  monotone surface such that
$\wedge^{(n)}_x \le \psi_{u_n}(x)$ for every $x\in U$. Let $\tilde
\bbP$ denote the law of the auxiliary monotone surface dynamics in
$\wt U$, starting at time $t_n$ from $\wedge^{(n)}$ and with
boundary conditions $\wedge^{(n)}_{\partial \wt U}$.
By monotonicity and the definition of $\wedge^{(n)}$ we have
\[
\bbP\left(A_x;\ \phi^\wedge(s)\le
  \psi_{u_n}\mbox{\;for every\;} s\in[t_n,L^3]\right)\le \tilde \bbP\left(\exists t\in[t_{n+1},L^3]\mbox{\;such that\;}\phi_x^{\wedge^{(n)}}(t)\ge \psi_{u_{n+1}}(x)\right).
\]
As in the proof of Theorem \ref{th:3d}, Propositions \ref{prop:quasisoff}
and \ref{prop:wilson} show that after time
\[
t_{n+1}-t_n=R_{n+1}(\log L)^{17/2}\ge const\times\text{diam}(\tilde U)^2 (\log L)^7
\]
 the dynamics $\tilde \bbP$ has
a variation distance of order $\exp(-c(\log L)^{3/2})$ for some $c>0$
from  its equilibrium $\pi_{\tilde U}^{\wedge^{(n)}}$ (recall that
$c_-\log L\le \log R_n\le c_+ \log L$, cf. Remark \ref{rem:odg}).
Theorem \ref{th:eq3d} gives that 
\[
\pi_{\tilde U}^{\wedge^{(n)}}[\phi_x\le \psi_{u_{n+1}}(x)]\ge 1-O\left( e^{-c(\log L)^{3/2}}\right).
\]
Indeed, % the distance between $x$ and $\bar x$ is smaller than $L$ and 
the point $(x,\psi_{u_{n+1}}(x))$ is at distance $(1+o(1))(\log L)^{3/2} $
from the plane $\tilde\Pi'$ containing the ``planar''  boundary 
condition $\wedge^{(n)}_{\partial \tilde U}$.

Putting everything together (plus a union bound over times
$t\in[t_{n+1},L^3]$ as in \eqref{eq:unbou}) one gets 
\[
\bbP(A_x;\ \phi^\wedge(s)\le
  \psi_{u_n}\mbox{\;for every\;} s\in[t_n,L^3])= O\left( e^{-c(\log L)^{3/2}}\right)\times O(
  L^3\times L^2)=O\left( e^{-\frac c2 (\log L)^{3/2}}\right).
\]
Finally, provided that we choose $c'=c/2$, from \eqref{eq:gathero} we get
\begin{gather*}
   \bbP(\cup_{x\in U}A_x)\le (n+1) e^{-c'(\log L)^{3/2}}
\end{gather*}
(the union bound over $x\in U$ gives just an extra $O(L^2)$) which is the desired claim.
 \end{proof}

\section{SOS model: Proof of Theorem \ref{th:sos}}

The proof of Theorem \ref{th:sos} is based on the following crucial results. The first is essentially an application of the results in \cite{MS}\footnote{The results in \cite{MS} are stated for the
  \emph{discrete} time chain, and thus a trivial overall factor of $L$
  must be taken into account when comparing our results with theirs. Moreover, \cite{MS} only deals with the case $h=0$.}, which can be formulated as follows. 
  %\subsection{A key estimate}\label{keyest}
 Let $L$ and $0\leq h\leq L$ be fixed and define the set $\O^\kappa_{L,h}$ of configurations $\xi\in\O_{L,h}$ satisfying
\begin{equation}
  \label{eq:key1}
\big|\xi_i-\bar\eta_i\big|
 \leq \sqrt{L}\,(\log L)^{\kappa}, \qquad i=1,\dots,L\,,
\end{equation}
where, as before, $ \bar\eta_i=\lfloor ih/(L+1)\rfloor$.
\begin{prop}
\label{th:key}
There exist $\alpha_1>0$ and $c>0$ such that the  following holds for every $\kappa\ge 1$. 
For any $L$  sufficiently large (how large depending on $\kappa$), uniformly in $0\leq h\leq L$ and $\xi\in\O^\kappa_{L,h}$:
\begin{equation}
  \label{eq:key2}
\|\mu_t^\xi-\pi\|  \leq e^{-c(\log L)^{2}},\quad \forall t\geq L^2\,(\log L)^{\alpha_1}.
\end{equation}
Moreover, the same statement holds for the evolution constrained to
stay above the wall
$\bar \eta$
(in this case of course one must require also that $\xi\ge\bar \eta$).  \end{prop}
%\begin{proof}

From the proof it will follow that one can actually choose $\alpha_1=17$ (with room to spare).
The second crucial result deals with the relaxation of the extremal
evolutions (the analogous result for monotone surfaces is Proposition \ref{prop:maxmin}). 
Let $\eta^{\wedge,+}(t)$ denote the evolution of the maximal initial
configuration $\wedge\equiv L+h$,
with the wall constraint  $\eta_i\geq \bar\eta_i$, $i=1,\dots,L$. 
Similarly, let $\eta^{\vee,-}(t)$ denote
the evolution of the minimal initial configuration $\vee\equiv -L$, with the wall constraint 
$\eta_i\leq \bar\eta_i,\ i=1,\dots,L$.
\begin{prop}\label{prop1}
Let $\alpha_1>0$ be as in Proposition \ref{th:key}.
For all $L$ sufficiently large, and for some time $T_1=O(L^2\,(\log L)^{4+\alpha_1})$:
\begin{equation}  \label{eq:part1}
\bbP\big(\exists i\in\{1,\dots,L\}\,,\;
\eta_i^{\wedge,+}(T_1)>\bar\eta_i+\sqrt{L}\,(\log L)^{5}\big)\leq L^{-3},
%,\qquad T:=L^2\,(\log L)^{c_3}\,.
\end{equation}
and similarly, 
\begin{equation}  \label{eq:part01}
\bbP\big(\exists i\in\{1,\dots,L\}\,,\;
\eta_i^{\vee,-}(T_1)<\bar\eta_i-\sqrt{L}\,(\log L)^{5}\big)\leq L^{-3}.
%,\qquad T:=L^2\,(\log L)^{c_3}\,.
\end{equation}
\end{prop}
Once Proposition \ref{th:key} and Proposition \ref{prop1} are established, it is an easy task to complete the proof of Theorem \ref{th:sos}:

\smallskip

\noindent
{\em Proof of Theorem \ref{th:sos}}. 
%\subsection{From Proposition \ref{prop1} to Theorem \ref{th:sos}}
From a standard comparison estimate, cf.\ also \eqref{eq:11} above, one has
$$
\max_{\eta}\|\mu_t^\eta-\pi\| %\leq CL^2\|\mu_t^\wedge-\mu_t^\vee\|
\leq CL^2\max\big(\|\mu_t^\wedge-\pi\|,\|\mu_t^\vee-\pi\|\big)\,,
$$
where $\mu_t^\wedge,\mu_t^\vee$ denote the law of evolutions 
$\eta^\wedge(t),\eta^\vee(t)$ from maximal and minimal initial
condition respectively, with no wall constraint. Let $G$ denote the event that $|\eta_i-\bar\eta_i|\leq\sqrt{L}(\log L)^5$ for all $i$. Observe that for any event $A$, for $t>T_1$ (where $T_1$ is as in Proposition \ref{prop1}) %:=L^2(\log L)^{c}$:
\begin{align*}
|\mu_t^\wedge(A)-\pi(A)|&
\leq |\bbP(\eta^\wedge(t)\in A\,;\, \eta^\wedge(T_1)\in G) - \pi(A)| + \bbP(\eta^\wedge(T_1)\notin G) \\
& \leq \max_{\xi\in  G} \|\mu_{t-T_1}^\xi - \pi\| + 2\bbP(\eta^\wedge(T_1)\notin G).
\end{align*}
 By monotonicity one can couple the dynamics in such a way that $\eta^{\vee,-}(T_1)
 \leq \eta^\wedge(T_1)\leq \eta^{\wedge,+}(T_1)$. Thus, from Proposition \ref{prop1}, 
 $\bbP(\eta^\wedge(T_1)\notin G) \leq L^{-3}$ for all $L$ large enough. 
 On the other hand, from Proposition \ref{th:key} one has
 $
 \max_{\xi\in  G} \|\mu_{T}^\xi - \pi\| \leq e^{-c(\log L)^{2}}$, for $T:=L^2(\log L)^{\alpha_1}.
 $
 Thus, taking $t=T_1+T$, it follows that 
 $
 %\max_{\sigma}\|\mu_t^\sigma-\pi\|
 \|\mu_t^\wedge-\pi\|\leq 3L^{-3}\,,
 $
 for all $L$ large enough. 
The same bound, by symmetry, can be obtained for $\|\mu_t^\vee-\pi\|$. Therefore, $$
 \max_{\eta}\|\mu_t^\eta-\pi\|\leq 3CL^{-1}\,,
 $$
which concludes the proof of Theorem \ref{th:sos}, with any power 
$\alpha>\alpha_1+4$. \qed

%%%%%%%

\subsection{Proof of Proposition \ref{prop1}}\label{proof:sos}
This and the next subsection are devoted  to the proof of Proposition \ref{prop1} (by symmetry it is sufficient to prove \eqref{eq:part1} only) assuming the validity of Proposition \ref{th:key}.
The latter is proved in Section \ref{keyest} below.
We first recall a basic
equilibrium estimate. 
\begin{lemma}\label{lem:eq}
There exists some constant $C>0$ such that, uniformly in $0\leq h\leq L$ and $H\geq \sqrt{L}\log L$:
\begin{equation}  \label{eq:eq}
\pi\big[\exists \,i=1,\dots,L:\; \big|\xi_i-\bar\eta_i\big|>H\big] \leq C\exp{(-C^{-1}\min\{H^2/L,H\})}.
%\pi(B)\leq C\,L\,\exp(-(\log L)^2/C)\,,
\end{equation}
%for all $H>0$. 
Moreover, \eqref{eq:eq}  continues to hold as it is for the interface
conditioned to stay above the wall, as well as for the unbounded SOS model
defined in \eqref{eq:unbo}. %Section \ref{sec:gap}.
\end{lemma}
The proof of Lemma \ref{lem:eq} 
%\eqref{eq:eq} 
can be obtained by standard large deviation arguments as in e.g.\ \cite[Appendix C]{MS}. 
We give a detailed proof of Proposition \ref{prop1} in the case $h=0$ only. The reader may verify that the same approach with minor modifications gives a proof for all $0\leq h\leq L$. 
In order to underline the similarities
between the proof of Proposition \ref{prop1} with that of the companion result
Proposition \ref{prop:maxmin} in the monotone surface case, we follow
closely the notation introduced there. 

Let $\rho_L=L\log L$ and, for $u>0$, let $\cC_{u}\subset \bbR^2$, denote 
the circular segment of height $u$ and base a segment  on the
horizontal axis, with length $2\rho_L$ and
centered at $L/2$.
% (see Figure \ref{circleseg}). 
%Here and below 
We shall use the notation $\eta\in \cC_{u}$ for any configuration $\eta$
such that $0\leq \eta_i\leq \psi_u(i),
i=1,\ldots,L$ where $\psi_u(i)=\sup\{y\in \bbR:\ (i,y)\in \cC_u\}$. 
%(see
%Figure \ref{circleseg}). 
Note that the radius of curvature $R$ of the circular segment $\cC_{u}$ satisfies
$(2R-u)u=\rho_L^2$. As in Section \ref{3d}, in what follows we will
always have $u\ll \rho_L\ll R$ so that $R=\rho_L^2/(2u)(1+o(1))$. 
% \begin{figure}[h]
% \centerline{
% \psfrag{ell}{$\ell$}
% \psfrag{hi}{$\psi_i(\ell,v)$}
% \psfrag{i}{$i$}
% \psfrag{v}{$v$}
% \psfrag{R}{$R$}
% \psfrag{ellh}{$\sqrt{\ell}(\log\ell)^{4}$}
% \includegraphics[scale=1]{mon-surf}}
% \caption{The region $\bar\cC_{v}$ obtained as the union of the $\ell\times \sqrt{\ell}(\log\ell)^{4}$ rectangle and the circle segment $\cC_{\ell,v}$ with chord length $\ell$ and height $v$.}
% \label{circleseg}
% \end{figure}
Define recursively $u_0=2L, t_0=0$ and 
$$
u_{n+1}=u_n-(\rho_L^2/u_{n})^{1/3}(\log L)^2\,,\quad
t_{n+1} = t_n + (\rho_L^2/u_n)^{4/3}(\log L)^{3+\alpha_1}$$
and let $M=\min\{n:\ u_n\le  \sqrt{L}(\log L)^4\}$. Clearly, if $R_n$
denotes the radius of curvature of the circular segment $\cC_{u_n}$,
then $R_n=\rho_L^2/(2u_n)(1+o(1))$. 

Crucially, as in the monotone surface case, $t_M=\tilde O(L^2)$.
\begin{lemma}
For any $L$ large enough   $t_M\leq L^2\,(\log L)^{4+\alpha_1}$.
\end{lemma}
\begin{proof}
We write
\begin{equation}\label{tenn}
t_M=\sum_{n=0}^{M-1} (\rho_L^2/u_n)^{4/3}(\log L)^{3+\alpha_1}\,.
\end{equation}
Next observe that, by convexity, $a^{4/3}-b^{4/3}\geq b^{1/3}(a-b)$ for all $a>b>0$. Taking $a=u_n$ and $b=u_{n+1}$, and using $u_{n}-u_{n+1}=(\rho_L^2/u_n)^{1/3}(\log L)^2$ one has
$$
u_{n}^{4/3} - u_{n+1}^{4/3}\geq u_{n+1}^{1/3}(u_{n}-u_{n+1}) \geq \frac12 L^{2/3}(\log L)^{8/3}\,,
$$
where the last bound follows from $u_{n+1}\geq \frac12 u_n$ which is easily seen to be implied by the assumption 
$u_n\geq \sqrt{L}(\log L)^4$ for all $0\leq n \le M-1$. Therefore,
\begin{gather*}
u_{n}^{4/3}= u_{M-1}^{4/3}+\sum_{m=n}^{M-2}(u_{m}^{4/3} - u_{m+1}^{4/3})\\\geq 
u_{M-1}^{4/3}+\frac12(M-n-1)L^{2/3}(\log L)^{8/3}\geq\frac12(M-n)L^{2/3}(\log L)^{8/3} \,,
\end{gather*}
the last bound following from $u_{M-1}^{4/3}\ge L^{2/3}(\log
L)^{8/3}$.
In conclusion, using this bound in \eqref{tenn}, 
$$
t_M\leq 2L^{2}(\log L)^{3+\alpha_1}\sum_{n=0}^{M-1}\frac1{M-n}\leq 2L^2\,(\log L)^{4+\alpha_1}\,,
$$
whenever $L$ is large enough. 
\end{proof}
\begin{rem}
\label{scale}As the careful reader has noticed, the length and time scales
$(u_n,t_n)$ are quite different from their analogue in the monotone
surface case. The main reason is the different order of magnitude of
the maximal equilibrium fluctuations in the two models: $\log L$ versus
$\sqrt{L}$. However their value is determined by the common recipe
which was described in Section \ref{sec:scale}. \end{rem}
The key step in the proof of \eqref{eq:part1} is analogous to
Proposition \ref{lemma:duro}.
\begin{prop} \label{propC}
For any $L$ large enough and
  for every $0\le n\le M$, with probability at least $1-n\,e^{-(\log L)^{3/2}}$,
  \begin{eqnarray*}
    \eta^\wedge(t)\in\cC_{u_n}
    \mbox{\;for every\;} t\in [t_n,L^3].
  \end{eqnarray*}
\end{prop}
If we now apply Proposition \ref{propC} with $n=M$ and use the fact
that $t_M=O(L^2\,(\log L)^{4+\alpha_1})$ and $u_M\le \sqrt L (\log
L)^5$, we obtain the desired claim  \eqref{eq:part1}.
\qed

\subsection{Proof of Proposition \ref{propC}}
\label{proof of propC}

As in the proof of Proposition \ref{lemma:duro} we proceed by
induction in $n\le M$. The initial step $n=0$ is obvious because the
maximal configuration $\wedge$ is inside $\cC_{u_0}$. 
Thus let us assume the statement true for $n<M$ and let us prove it
for $n+1$.

For $1\le i\le L$ define
the event
\[
A_i=\{\exists\, t\in [t_{n+1},L^3]:\eta^\wedge_i(t)> \psi_{u_{n+1}}(i)\}.
\]
Using the inductive assumption we may write
\begin{gather}
%\nonumber
 \bbP(\cup_{i=1}^L A_i)
\label{eq:gatherosos}
\le \sum_{i=1}^L\bbP(A_i;\ \eta^\wedge(s)\in \cC_{u_n}
\mbox{\;for every\;} s\in[t_n,L^3])+n \,e^{-(\log L)^{3/2}}.
\end{gather}
Fix $i=1,\ldots,L$ and consider the line $\mathbb L$ tangent to
$\cC_{u_n}$ at the point
$(i,\psi_{u_n}(i))$ and the line  $\mathbb L'$
obtained by translating downwards (i.e. in the $-y$ direction)
$\mathbb L$  by
\[2(u_n-u_{n+1})= 2(\rho_L^2/u_{n})^{1/3}(\log L)^2.\]
Let us denote by $x_-,x_+$ the horizontal coordinates of the leftmost and
rightmost  points of $\mathbb L'\cap \cC_{u_n}$ and let $I$ be the set
of integers in $[x_-,x_+]$. Clearly $|I| = O((\rho^2_L/u_n)^{2/3}\log
L)$. 

Consider now the SOS dynamics in the interval $I$ with boundary conditions
equal to $\lfloor \psi_{u_n}(i_\pm)\rfloor $ at the left and right boundary $i_\pm$ of
$I$ and floor at zero height. This auxiliary evolution starts at time $t_n$ from the maximal
configuration $\wedge^{(n)}$ in the set of $\eta\in[0,L]^I$ such that $\eta_i\le \psi_{u_n}(i)$
for any $i\in I$.  We denote by $\bbP'$ the law of this auxiliary chain.
  Observe that $\wedge^{(n)}$ is within distance
$2(u_n-u_{n+1})= O(\sqrt{|I|}(\log |I|)^{3/2})$   from
the line $\mathbb L'$ so that Proposition \ref{th:key} will be
applicable with $\kappa=3/2$.
By monotonicity we have
\[
\bbP\left(A_i;\ \eta^\wedge(s)\in \cC_{u_n} \mbox{\;for every\;} s\in[t_n,L^3]\right)\le \bbP'\left(\exists t\in[t_{n+1},L^3]\mbox{\;such that\;}\eta_i^{\wedge^{(n)}}(t)\ge \psi_{u_{n+1}}(i)\right).
\]
Because of Proposition \ref{th:key}, after time
$|I|^2(\log|I|)^{\alpha_1}\le t_{n+1}-t_n$ the dynamics $ \bbP'$ has
a variation distance of order $\exp(-c(\log L)^{2})$ from  its
equilibrium which we denote by $\pi_I^{(n)}$. Since the distance
between the point 
$(i,\psi_{u_{n+1}}(i))$ and the line $\mathbb L'$ is at least $c
\sqrt{|I|}(\log |I|)^{3/2}$ for some $c>0$,   
Lemma \ref{lem:eq} 
gives that 
\[
\pi_I^{(n)}\left(\eta_i\le \psi_{u_{n+1}}(i)\right)\ge 1-O\left(
  e^{-c(\log |I|)^{3}}\right)\ge 1- e^{-c' (\log L)^3}.
\]
As in the proof of Proposition \ref{lemma:duro}, simple union bounds
over $i\in [1,L]$ and $t\in[t_{n+1},L^3]$ give 
\[
\sum_{i=1}^L\bbP(A_i;\ \eta^\wedge(s)\in \cC_{u_n}
\mbox{\;for every\;} s\in[t_n,L^3]) \le L^5 e^{-c''(\log L)^2}\le
e^{-(\log L)^{3/2}}
\]
which finishes the proof of the inductive step.
\qed

%\begin{figure}[b]
%\centerline{
%\psfrag{ell'}{$\ell'$}
%\psfrag{Pi}{$P_i$}
%\psfrag{i-}{$i_-$}
%\psfrag{i}{$i$}
%\psfrag{i+}{$i_+$}
%\psfrag{v}{$v$}
%\psfrag{v'}{$v'$}
%\psfrag{R}{$R$}
%\psfrag{ellh}{$\sqrt{\ell}(\log\ell)^4$}
%\psfig{file=SAVE.eps,height=1.5in,width=4.0in}}
%\caption{The circle segment with parameters $(\ell',v')$
%and the points $i_-$ and $i_+$ determined by it. The picture corresponds to the case where $i$ satisfies \eqref{centralregion}.}
%\label{circleseg1}
%\end{figure}

%\begin{figure}[h]
%\centerline{
%\psfrag{ell'}{$\ell'$}
%\psfrag{Pi}{$P_i$}
%\psfrag{i-}{$i_-$}
%\psfrag{i}{$i$}
%\psfrag{i+}{$i_+$}
%\psfrag{v}{$v$}
%\psfrag{v'}{$v'$}
%\psfrag{R}{$R$}
%\psfrag{h*}{$h_*$}
%\psfig{file=cricleseg2.eps,height=1.5in,width=4.0in}}
%\caption{The circle segment with parameters $(\ell',v')$ in the case where $i$ satisfies $i\in[0,\frac12,\ell']$. }
%\label{circleseg2}
%\end{figure}

\subsection{Proof of Proposition \ref{th:key}}\label{keyest}
We describe the main steps in the absence of the wall. At the end we
will comment on the case with the wall. 

Let $\xi^{\wedge_\kappa}(t),\xi^{\vee_\kappa}(t)$ denote the Glauber dynamics at time $t$
starting from the maximal $\wedge_\kappa$ and
minimal $\vee_\kappa$ initial
condition in $\O^\kappa_{L,h}$, respectively,
and let $\mu_t^{\wedge_\kappa},\mu_t^{\vee_\kappa}$ denote
their distribution. Notice that
  $\wedge_\kappa,\ \vee_\kappa$ are different from the maximal $\wedge$ and minimal
  $\vee$ configuration  in $\O_{L,h}$. For any $\xi \in \O^\kappa_{L,h}$, for any event $A\subset \O_{L,h}$ one has:
  $$\mu_t^\xi(A)-\pi(A) =
  \sum_{\eta\in\O_{L,h}} \pi(\eta)( \mu_t^\xi(A)-\mu_t^\eta(A)).
  $$
 Therefore, 
%\begin{gather*}
$$  \max_{\xi \in \O^\kappa_{L,h}}\|\mu_t^\xi-\pi\|\le
  \pi(\O_{L,h}\setminus \O^\kappa_{L,h})+ \max_{\xi,\,\eta \in
    \O^\kappa_{L,h}}\|\mu_t^\xi-\mu_t^\eta\|.
$$
    %\\ \le
%\pi(\O_{L,h}\setminus \O^\kappa_{L,h}) +
%\bbP(\xi^{\wedge_\kappa}(t)\neq \xi^{\vee_\kappa}(t))\\
%\le Ce^{-(\log L)^2/C} + \bbP(\xi^{\wedge_\kappa}(t)\neq \xi^{\vee_\kappa}(t))
%\end{gather*}
By monotonicity, $\max_{\xi,\,\eta \in
    \O^\kappa_{L,h}}\|\mu_t^\xi-\mu_t^\eta\|\leq \bbP(\xi^{\wedge_\kappa}(t)\neq \xi^{\vee_\kappa}(t))$, 
where $\bbP(\cdot)$ denotes the global monotone coupling (see Section \ref{sec:monoton}).
Using 
\eqref{eq:eq} to estimate $\pi(\O_{L,h}\setminus \O^\kappa_{L,h})$ for any $\kappa\geq 1$, we arrive at 
$$
\max_{\xi \in \O^\kappa_{L,h}}\|\mu_t^\xi-\pi\|\le
Ce^{-(\log L)^2/C} + \bbP(\xi^{\wedge_\kappa}(t)\neq \xi^{\vee_\kappa}(t)).
$$
Let now $\nu_+:=\pi(\cdot \tc \xi \ge \wedge_\kappa)$ and
$\nu_-:=\pi(\cdot \tc \xi \le \vee_\kappa)$. Monotonicity implies that 
\begin{gather*}
\bbP(\xi^{\wedge_\kappa}(t)\neq \xi^{\vee_\kappa}(t))\le
\sum_{i=1}^L\sum_{j=-L}^{L+h}\bbP(\xi_i^{\wedge_\kappa}(t)\ge j\,;\,\xi_i^{\vee_\kappa}(t)<j)= 
\sum_{i=1}^L\sum_{j=-L}^{L+h}\left[\mu_t^{\wedge_\kappa}(\xi_i\ge j)-\mu_t^{\vee_\kappa}(\xi_i\geq
  j)\right]\\
\le 2L(L+h) \left(\|\mu_t^{\nu_+}-\pi\|+\|\mu_t^{\nu_-}-\pi\|\right)=4L(L+h)\|\mu_t^{\nu_+}-\pi\| 
\end{gather*}
where $\mu_t^{\nu_\pm}$ is the law of the process at time $t$ starting
from the distribution $\nu_\pm$. Above we used symmetry to conclude that
$\|\mu_t^{\nu_+}-\pi\|=\|\mu_t^{\nu_-}-\pi\|$. 

Since the
event $\{\xi \ge \wedge_\kappa\}$ is increasing, one has 
$\pi \preceq \nu_+$ and the relative density $\nu_+(\xi)/\pi(\xi)$ is an increasing
function. The Peres-Winkler censoring argument \cite{notePeres}
(see also Lemma 2.2 in \cite{MS}) proves that 
the same holds if we replace $\nu_+$ with $\mu_t^{\nu_+}$. Thus the
event $F=\{\mu_t^{\nu_+}(\xi)\ge \pi(\xi)\}$ is increasing and 
$\|\mu_t^{\nu_+}-\pi\|=\mu_t^{\nu_+}(F)-\pi(F)$. 

Next, we claim that, for 
some constants $\alpha_1\le 17$ and $c>0$, for all $t\ge L^2\,(\log
L)^{\alpha_1}$ and for any
    increasing event $A$,
\begin{equation}
  \label{eq:key3}
  \mu_t^{\nu_+}(A)-\pi(A)\le e^{-c(\log L)^2}.
\end{equation}
%for some constant $c>0$.
Clearly \eqref{eq:key3} finishes the proof of Proposition \ref{th:key}. In turn, the proof of
\eqref{eq:key3} is adapted from the proof of a similar result in \cite{MS}
(see Lemma 4.5 there). For the reader's convenience, we now
describe its main steps.\\

\noindent
{\em Step 1.} One first introduces an auxiliary \emph{parallel column dynamics} which consists
in replacing the single-site moves of the Glauber dynamics \eqref{ppm}
with \emph{non-local moves} as follows. At each integer time
$s=1,2,\dots$ firstly all even-numbered heights
$\{\xi_{2i}\}$, $i=1,2,\dots$, are re-sampled from their equilibrium distribution given the
neighboring odd-numbered heights $\xi_{2i+1}$, $i=0,1,\dots$, and then
viceversa with the role of even/odd columns reversed. 
By construction the
column chain is ergodic and reversible w.r.t. $\pi$.  Moreover, and
that is %really 
the main reason for introducing it, 
%the so-called 
Wilson's argument \cite{Wilson} applies to the new chain. Namely,
by an almost exact computation, one can prove  that under the column dynamics the maximal and minimal configuration in $\O_{L,h}$
couple in a time $O(L^2\log L)$ (see Section 3 in
\cite{MS}). \\

\noindent
{\em Step 2.} Next, one relates the
single site Glauber dynamics started from $\nu_+$ to the column
dynamics described above by means of the
Peres-Winkler censoring idea. More precisely, given a time-lag $T$,
one splits the time interval $[0,t]$ into $N=t/2T$ epochs each of
length $2T$ (we assume for simplicity $N\in \bbN$) and in each epoch one first censors (i.e.\ freezes) all the Poisson clocks
at odd sites for the first half of the epoch and then all the even
sites for the second half. If $\tilde \mu_t^{\nu_+}$ denotes the
distribution of the censored chain, then the censoring inequality  of
\cite{notePeres} says that $\mu_t^{\nu_+}\preceq \tilde
\mu_t^{\nu_+}$ so that, for any increasing event $A$, 
\[
\mu_t^{\nu_+}(A)-\pi(A)\le \tilde \mu_t^{\nu_+}(A)-\pi(A)\le \|\tilde
\mu_t^{\nu_+}-\pi\|.
\] 
If the free parameter $T$ is chosen in
such a way that each epoch simulates very closely one step of the
column chain then, thanks to Step 1, 
\[
\|\tilde\mu_t^{\nu_+}-\pi\|\le e^{-c(\log L)^2}
\]
provided that $N\geq L^2(\log L)^3$.  

As explained in
\cite{MS}, the overhead $T$ introduced in the mixing time by the censoring is essentially
the square of the maximum gradient $|\nabla\xi|_\infty:=\max_i |\xi_{i+1}-\xi_i|$
that may arise in the censored dynamics. Since gradients of the
equilibrium interface are  i.i.d.\ 
random variables with exponential tail, conditioned
to their sum being equal to $h$, well known estimates 
(see e.g.\ 
Lemma \ref{lem:eq} above 
or \cite[Appendix C]{MS}) imply that
for any $i=0,\dots,L$, uniformly in $0\leq h\leq L$, and $\ell\geq (\log L)^2$:
$$
\pi\left(|\nabla\xi|_\infty\ge \ell\right)\leq Ce^{-\ell/C}.
$$
Therefore, 
for any event $A$ and starting from
the distribution $\nu:=\pi(\cdot \tc A)$, %reversibility 
the invariance of $\pi$ %immediately
implies that 
\[
\mu_t^\nu( |\nabla\xi|_\infty\ge \ell)\le
\frac{\pi\left(|\nabla\xi|_\infty\ge \ell\right)}{\pi(A)}\le C\,
%L^{3/2}
\frac{e^{-\ell/C}}{\pi(A)},
\]
for $\ell \geq (\log L)^2$.
A similar bound holds for the censored dynamics $\tilde \mu_t^\nu$. The above
observations lead to the following lemma\footnote{Actually in \cite{MS}
    the bound is given as $o(1/L^b)$ for any $b>0$ but it is easily seen to hold in the form 
    stated here.}.
\begin{lemma}[see Lemma 4.2 in \cite{MS}]
\label{Lemma1}
%For any $\gamma$ large enough the following holds. 
Let $A$ be any increasing event, and consider the censored single-site dynamics started from 
$\nu:= \pi(\cdot \tc A)$. Let $D := 1+\log(\frac 1{\pi(A)})$ and choose
$T=D^2(\log L)^9$ and $t= T L^2(\log L)^{3}$.  
Then, for some constant $c>0$ and for all $L$ large enough 
\begin{equation*}
%  \label{eq:key 1}
\|\tilde \mu^{\nu}_{t}-\pi\|\le e^{-c(\log L)^2}.     
\end{equation*}
\end{lemma}
\noindent
{\em Step 3.} The above lemma cannot be applied directly to our case
$A=\{\xi\ge \wedge_\kappa\}$ because at time $t=0$ the maximum
gradient is too large, of order
$\sqrt{L}\,(\log
  L)^{\kappa}$, and therefore $\pi(A)= O(e^{-c \sqrt{L}\,(\log
  L)^{\kappa}})$ is very small. This makes the corresponding overhead $T$ too
large.  The solution to this problem proposed in \cite{MS} goes as
follows (see also Lemma 4.5 and Corollary 4.6 there).
\begin{lemma}\label{lem:recur}
Fix $\kappa \geq 1$. %Then for any $\alpha$ large enough the following holds. 
Let
$(\log L)^2\le H\le \sqrt{L}(\log L)^\kappa$ and let $A_H=\{\xi_i \ge
\bar\eta_i + H,\ i=1,\dots L\}$. Let $\nu_H:=\pi(\cdot \tc A_H)$ and let $t=
L^2(\log L)^{16}$. Then,  for some $c>0$, for $L$ large enough, 
\begin{equation}\label{recur}
\mu_t^{\nu_H}(f)\le \nu_{H/2}(f) + e^{-c(\log L)^2},
\end{equation}
for any non-negative, increasing function $f$ with
$\|f\|_\infty\le 1$. 
\end{lemma}
\begin{proof}
A sketch of the proof is as follows. Choose $\ell = H^2/(\log L)^2$
and consider the Glauber dynamics on the enlarged state space $\O_{\ell,L,h}$ of SOS
interfaces $\{\eta_i\}_{i=-\ell}^{L+\ell}$ such that  
\begin{equation*}
  \begin{cases}
  \eta_i\in [-L,L+h] & \text{ if   $i\in [1,L]$}\\
  \eta_i\geq \lfloor ih/(L+1)\rfloor & \text{ if  $i\in
  [-\ell,0]\cup [L+1,L+\ell]$}
  \end{cases}
     \end{equation*}
    and with boundary condition
    \[
\eta_{-(\ell
  +1)}= -\lfloor (\ell+1)h/(L+1)\rfloor,\  \eta_{L+\ell+1}=\lfloor(L+\ell+1)h/(L+1)\rfloor \,.
\]
Let $\pi^{(\ell)}$ be the corresponding SOS reversible equilibrium
measure and let $\mu_t^{\nu_{\ell,H}}$ be the distribution at time $t$
starting from $\nu_{\ell,H}:=\pi^{(\ell)}(\cdot\tc A_H)$. Clearly the marginals of
$\mu_t^{\nu_{\ell,H}}$ and of $\pi^{(\ell)}$ on $\O_{L,h}$ are
stochastically larger than $\mu_t^{\nu_H}$, and $\pi$ respectively. Moreover, from standard local limit theorem estimates (see e.g.\ \cite[Appendix C]{MS}) one has 
\begin{equation}\label{lowloc1}
\pi^{(\ell)}(A_H)\ge \frac{1}{CL^\gamma}\,e^{-CH^2/\ell}\ge e^{- C'(\log L)^2}
\end{equation}
for suitable constants $C,C',\gamma>0$. Then, we apply Lemma \ref{Lemma1} to the enlarged dynamics, with $D=O((\log L)^2)$, and get that at time $t=L^2(\log L)^{b}$, with $b=16$, 
\[
\|\mu_t^{\nu_{\ell,H}}-\pi^{(\ell)}\|\le e^{-c(\log L)^2}.
\]
Thus, for any $f$ as in the lemma, 
\[
\mu_t^{\nu_H}(f)\le \mu_t^{\nu_{\ell,H}}(f)\le \pi^{(\ell)}(f) +
e^{-c(\log L)^2}.
\]
Finally, let $E_{H}=\{\eta\in\O_{\ell,L,h}:\ \{\eta_1\le H\}\cap
\{\eta_L\le h+H\}\}$. Then, using the positivity of $f$, 
\begin{gather*}
\pi^{(\ell)}(f) \le \pi^{(\ell)}\bigl(f\tc E_{H/2}\bigr)
  +\pi^{(\ell)}\bigl(E^c_{H/2}\bigr)
\le \pi(f\tc A_{H/2}) +\pi^{(\ell)}\bigl(E^c_{H/2}\bigr). 
\end{gather*}
The Gaussian bounds of Lemma \ref{lem:eq} show that 
\[
  \pi^{(\ell)}\bigl(E^c_{H/2}\bigr)\le C  \,e^{-H^2/C\ell}\le e^{-c(\log L)^2}.
\]
In conclusion, the desired estimate \eqref{recur} holds with $t=L^2(\log L)^{16}$.
\end{proof}
Using the semigroup property, one can iterate the bound of Lemma \ref{lem:recur} to go from the initial height $H_0:=\sqrt {L}(\log L)^\kappa$ to the final height $H_n:=(\log L)^2$, with $n=O(\log L)$
steps. %Set $\hat\nu:=\nu_{(\log L)^2}$. 
%A simple $O(\log(L))$ times iteration of the above result shows that
Therefore, at time $t= c' L^2(\log L)^{17}$, one has, for all $f$ as in Lemma \ref{lem:recur}:
\begin{equation}
\mu_{2t}^{\nu_{H_0}}(f)\le \mu_t^{\nu_{H_n}}(f) + e^{-c''(\log L)^2}\,,
\label{eq:key4}
\end{equation}
for some constant $c',c''>0$.
To finish the proof  Proposition \ref{th:key}, cf. \eqref{eq:key3}, observe that 
$\nu_{H_0}=\nu_+$ and therefore, by \eqref{eq:key4}, it is sufficient to establish  
% one is left with the estimate 
\begin{equation}
\mu_t^{\nu_{H_n}}(f)\le \pi(f) +
e^{-c_1(\log L)^2}\,,
%\mu_t^{\nu}(f)\le \nu_{(\log L)^2}(f) + e^{-c\log(L)^3}\,,
\label{eq:key5}
\end{equation}
for some $c_1>0$, for all $f$ as above. 
However, from the censoring inequality we know that 
$\mu_t^{\nu_{H_n}}(f)\le \tilde \mu_t^{\nu_{H_n}}(f)$. 
Moreover, a shift of the boundary height by $O((\log L)^2)$ shows that %as in \eqref{lowloc1},
\[
\pi\big(A_{H_n}\big)\geq e^{-C(\log L)^2}.
\]
Therefore,   
\eqref{eq:key5} follows from Lemma \ref{Lemma1}, with $D=O((\log L)^2)$, for any 
$t\geq D^2L^2(\log L)^{12} = O(L^2(\log L)^{16})$.
This ends the proof of Proposition \ref{th:key} without the wall, with
the choice of the constant $\alpha_1=17$. 

When the wall is present, the symmetry between the maximal and minimal
configuration breaks down and one has the new problem of proving
convergence to equilibrium within the correct time scale starting from
the minimal configuration $\bar \eta$. This problem has been solved in
\cite{MS} when $h=0$ but the same proof applies to $h\in [0,L]$.
\qed

\section{SOS model: Proof of Theorem \ref{th:gap}}
Let $\gap(L,H)$ denote the spectral gap of the SOS model with zero boundary conditions, floor at $-H$ and ceiling at $+H$. 
Similarly, let $\gap^w(L,H)$ denote the same quantity for the system with the wall, i.e.\ 
 the spectral gap of the SOS model with zero boundary conditions, floor at $0$ and ceiling at $+H$. 
\begin{prop}\label{propgap}
There exists $c>0$ such that for all $L\in\bbN$ sufficiently large, and for all $H\in\bbN$:
\begin{equation}  \label{propgap1}
\gap(L,H)\geq c\,L^{-2}(\log L)^{-\alpha}\,,
\end{equation}
where $\alpha>0$ is the same constant appearing in Theorem \ref{th:sos}. The same bound holds for $\gap^w(L,H)$. 
\end{prop}
The first observation is that Proposition \ref{propgap} immediately implies Theorem \ref{th:gap}. Indeed, this follows from an elementary approximation since the bound \eqref{propgap1} is uniform in $H$: 
%the generator $\cL$ given in \eqref{gener} is a bounded linear operator in $L^2(\O,\pi)$,  and for each $f\in L^2(\O,\pi)$, 
both 
$\cE(f,f)$ and $\Var(f)$ can be seen as limit as $H\to\infty$ of the corresponding quantities for the model with floor and ceiling at $\pm H$, so that for fixed $L$ one has $\gap(L)\geq \liminf_{H\to\infty}\gap(L,H)$. 

\subsection{Proof of Proposition \ref{propgap}}
To prove Proposition \ref{propgap} we use a recursive approach together with a powerful idea borrowed from the mathematical theory of kinetically constrained spin systems (see section 4 in \cite{CMRT} and the proof of Proposition \ref{recursive} below).   We give the details of the proof of the lower 
bound on $\gap^w(L,H)$, and will later illustrate the minor modifications necessary to cover the case without the wall; see Section \ref{nowall} below. 

For any rectangle $\Lambda$ in $\bbZ^2$, with base 
$\ell_\Lambda$ (horizontal side) and height $h_\Lambda$ (vertical side), we write 
$\pi_{\L}$ for the equilibrium measure of  the SOS model in the
rectangle $\L$ with zero boundary conditions at the endpoints of the
bottom horizontal side of length $\ell_\L$
(the lower horizontal side of the rectangle has vertical coordinate zero). We call $\O_\L$ the set of all allowed configurations, i.e.\ $\O_\L=\{0,\dots,h_\L\}^{\ell_\L}$. Let also 
\begin{equation}\label{gapnotat}
\gap(\L) = \inf_f \,\frac{\cE_\L(f,f)}{\Var_\L(f)},
\end{equation}
where we use the notation $\cE_\L$, $\Var_\L$ for the Dirichlet form and the variance associated to $\pi_\L$, and $f$ ranges over all functions such that $\Var_\L(f)\neq 0$. 
Note that, with this notation,  one has $\gap(\Lambda)=\gap^w(\ell_\Lambda,h_\Lambda)$.
Define
$$
\gamma(L,n) = \max_{\Lambda\in \bbF_{n,L}} \gap(\Lambda)^{-1},
$$
where $\bbF_{n,L}$ is the class of all rectangles $\Lambda$ in $\bbZ^2$ such that $\ell_\Lambda\leq L$
and $h_\Lambda\leq(3/2)^n$. The crucial recursive estimate reads as follows.

\begin{prop}
\label{recursive}
Let $n_0=n_0(L)\in\bbN$ be such that $(3/2)^{n_0}\le L<(3/2)^{n_0+1}$. Then %there exists $\alpha>0$ such that, 
for any $L$ large enough, for any $n\ge n_0$, 
\begin{equation}  \label{recursive1}
\gamma(L,n)\le (1+\delta_n)%(1+\epsilon_n)(1+\beta_n^{-1})
\,\gamma(L,n-1)
\end{equation}  
where $\delta_n=(3/2)^{-n/6}$. %$\epsilon_n= e^{-  (3/2)^{n/3}}$, $\beta_n=(3/2)^{n/5}$.
\end{prop}
Let us finish the proof of Proposition \ref{propgap}, assuming the validity of Proposition \ref{recursive}. Iterating the bound \eqref{recursive1} one has
\begin{equation}  \label{propgap2}
\gamma(L,n)\leq C\,\gamma(L,n_0)\,,
\end{equation}
and 
$$
C=
\prod_{n=n_0}^\infty (1+\delta_n)%(1+\epsilon_n)(1+\beta_n^{-1})
<\infty,
$$ 
uniformly in $L$.
Proposition \ref{propgap} now follows from \eqref{propgap2} and the following lemma.
%It remains to bound from above $\gamma^w_{n_0}$. 
\begin{lemma}
There exists $C<\infty$ such that for any $L$ large enough
\[
\gamma(L,n_0)\le CL^2(\log L)^\alpha\,,
\] 
with $n_0=n_0(L)$ as in Proposition \ref{recursive}.
\end{lemma}
\begin{proof}[Proof]
%For any rectangle $\L$ let $\ell_\L$ be the length of its base and write 
Note that,  thanks to the straightforward bound $\gap^{-1} \leq \tmix$
(which holds for any reversible Markov chain \cite{Peres}), if $h_\L\leq (3/2)^{n_0}\leq C\ell_\L\leq CL$ for some $C>0$,  
from Theorem \ref{th:sos} one has 
$$\gap(\L)^{-1} = O(L^2(\log L)^\alpha).$$
(Strictly speaking Theorem \ref{th:sos} is stated only for $h_\L\leq \ell_\L$, but the proof works as it is as soon as $h_\L\leq C\ell_\L$ for some $C$). 
Therefore, the same bound applies to every $\Lambda\in  \bbF_{n_0,L}$ such that 
$\ell_\Lambda > L/8$. Thus, one has to estimate only
$
\max\{
\gap(\Lambda)^{-1}\,,\;\Lambda\in \bbF_{n_0,L}:\;\ell_\Lambda \le L/8\}.
%\max_{\stackrel{\Lambda\in \bbF_{n_0,L}:}{\ell_\Lambda > L/8}}\gap(\Lambda)^{-1}    \big)
$
In order to deal with this term, define $f(L)=\gamma(L,n_0)$, for all $L$ and $n_0=n_0(L)$ as above.
%such that $(3/2)^{n_0}\leq L \leq (3/2)^{n_0+1}$.
%\[
%f(L):= \max\{\gap(\Lambda)^{-1}\,,\; \Lambda\in \bbF_{n_0,L}\}
%\]
From Proposition \ref{recursive}, one has
\[\max\{
\gap(\Lambda)^{-1}\,,\;\Lambda\in \bbF_{n_0,L}:\;\ell_\Lambda \le L/8\}
\leq C_1\, f(L/8)\,,
\]
 where the constant $C_1=\prod_{n=n_0-3}^{n_0} (1+\delta_n)
 %(1+\epsilon_n)(1+\beta_n^{-1})
 \leq 2$ for $n_0$ large enough. 
In conclusion,
\[
f(L) \le \max\Big\{2 f(L/8), CL^2(\log L)^{\alpha}\Big\}\le 2 f(L/8) + CL^2(\log L)^{\alpha}.
\]
A simple induction now proves that $f(L)\le C'  L^2(\log L)^{\alpha}$ for some new constant $C'$.
\end{proof}

\subsection{Proof of Proposition \ref{recursive}}\label{sec:rec}
Take $n\in\bbN$ and $\Lambda\in\bbF_{n,L}$. Without loss of generality we write
$\Lambda=[1,\ell_\Lambda]\times [0,h_\Lambda]$. We need to show that
$$
\gap(\Lambda)^{-1}\le (1+\delta_n)
%(1+\epsilon_n)(1+\beta_n^{-1})
\gamma(L,n-1).
$$
We use a geometric construction analogous to that of \cite[Proposition 3.2]{BCC}. 
 If $h_\Lambda\leq (3/2)^{n-1}$ then actually
$\Lambda\in\bbF_{n-1,L}$ and therefore $\gap(\Lambda)^{-1}\leq \gamma(L,n-1)$. If instead 
$h_\Lambda > (3/2)^{n-1}$, then it is easily seen that we can write $\Lambda=\Lambda_1\cup\Lambda_2$
where $\Lambda_1,\Lambda_2$ are such that:
\begin{enumerate}[(a)]
\item $\Lambda_i\in \bbF_{n-1,L},\ i=1,2$;
\item $\Lambda_i$, $i=1,2$,  has the same base $[1,\ell_\Lambda]$ of $\Lambda$; 
\item the overlap rectangle $I=\Lambda_1\cap 
\Lambda_2\neq
\emptyset$ has base length $\ell_\Lambda$ and height at least
 $\Delta_n:=(3/2)^{3n/4}$.    
\end{enumerate}
Moreover, there are at least $ s_n:=(3/2)^{n/5}$ such decompositions $\bigl\{(\Lambda_1^{(i)},\Lambda_2^{(i)})\bigr\}_{i=1}^{s_n}$ with the property
that the overlap rectangles $I^{(i)} = \Lambda^{(i)}_1\cap\Lambda^{(i)}_2$ are {\em disjoint}, i.e.\
$I^{(i)}\cap I^{(j)} = \emptyset$, for all $i\neq j$.
% corresponding to two distinct pairs
%$(\Lambda^{(i)}_1,\Lambda^{(i)}_2)$, and $(\Lambda^{(j)}_1,\Lambda^{(j)}_2)$, $i\neq j$, are {\em disjoint}.

Next, fix one of the $s_n$ decompositions mentioned above, say
$\Lambda_1=[1,\ell_\Lambda]\times[0,h_{\Lambda_1}]$ and
$\Lambda_2=[1,\ell_\Lambda]\times[h_{\Lambda_1}-|I|,h_\Lambda]$, where
$ h_\Lambda = h_{\Lambda_1} +h_{\Lambda_2} -h_I$, with  $h_I\ge
\Delta_n$
denoting the height of the overlap rectangle $I$.
Consider the distribution $\pi_{\Lambda_2}^\eta$ corresponding to the equilibrium measure in the region $\Lambda_2$ conditioned to the value of the configuration $\eta$ in the region $\Lambda\setminus\Lambda_2 = [1,\ell_\Lambda]\times[0,h_{\Lambda_1}-h_I]$. Let $$\cC(\eta):=\{i\in [1,\ell_\Lambda]:\ \eta_i > h_{\Lambda_1}-h_I\}.$$ 
Note that, if $\cC(\eta)\neq\emptyset$, then  $\cC(\eta)$ is the disjoint union of intervals $B_1,B_2,\dots , B_k$, and one has that $\pi_{\Lambda_2}^{\eta}$ is a product measure  
$\otimes_{i=1}^k\pi_{R_i}$ on the rectangles $R_i= B_i\times [h_{\Lambda_1}-h_I+1,h_\Lambda]$, such that each $\pi_{R_i}$ is the SOS equilibrium measure with floor at height $h_{\Lambda_1}-h_I+1$ and ceiling at height $h_{\Lambda}$. If, on the other hand, $\cC(\eta)=\emptyset$, then $\eta$ is entirely contained inside $\L\setminus\Lambda_2$, %(for shortness $\eta\subset \Lambda_1$), 
and $\pi_{\Lambda_2}^{\eta}$  is trivial, i.e.\ it gives full mass to the empty configuration in the region $\Lambda_2$.

Define a \emph{constrained block  dynamics} as follows. With rate one the current configuration $\eta$ is re-sampled inside $\Lambda_2$ from the distribution $\pi_{\Lambda_2}^\eta$ described above. 
Moreover, if $\eta\subset \L_1$, that is $\eta_i\leq h_{\L_1}$ for all $i=1,\dots,\ell_\L$, then with rate one, the current configuration $\eta$ is re-sampled inside $\Lambda_1$ from the distribution $\pi_{\Lambda_1}$ (defined just before \eqref{gapnotat}).
This process is {\em kinetically constrained} in that the region $\Lambda_1$ is updated only if $\eta\subset\Lambda_1$. 

It is not hard to check that the above dynamics is reversible w.r.t.\ 
the SOS equilibrium measure %in $\Lambda$, 
$\pi_\Lambda$, and that its Dirichlet form is equal to:
\[
\cE_{\rm block}(f,f) = \pi_\Lambda\left(1_{\{\eta\subset \Lambda_1\}}
\Var_{\Lambda_1}(f) + \Var^{\eta}_{\Lambda_2}(f)\right),
\]
where $\Var_{\Lambda_1}$ stands for the variance w.r.t.\ $\pi_{\Lambda_1}$, while $\Var^{\eta}_{\Lambda_2}$ stands for the variance w.r.t.\ $\pi_{\Lambda_2}^\eta$.
Notice that, even starting from the maximal configuration $\wedge$ in $\Lambda$, if we first update $\Lambda_2$ and then $\Lambda_1$, with very high probability we reach the distribution $\pi_{\Lambda_1}$. Indeed, thanks to the fact that $h_I\ge \Delta_n \gg \sqrt{\ell_\Lambda}$,  the first update in $\Lambda_2$ 
will produce with very high probability a new configuration entirely contained in $\Lambda_1$. 
In turn, since $\|\pi_{\Lambda_1}-\pi_\Lambda\|$ is very small (because $h_{\Lambda_1}\gg\sqrt{\ell_{\Lambda}}$), the above example suggests that the spectral gap of the constrained block dynamics should be very close to one. This is quantified in the next lemma.
\begin{lemma}
\label{gap-block}
For all $L$ large enough, and $n\geq n_0(L)$, for all functions $f$:
\[
\Var_\Lambda(f)\le \big(1+ e^{-(3/2)^{n/4}}\big)\,\cE_{\rm block}(f,f). %, \qquad \forall f\in L^2(\pi_\Lambda).  
\]
\end{lemma}
Assuming the validity of Lemma \ref{gap-block}, the proof of Proposition \ref{recursive} can be completed as follows. From \eqref{dirich} and \eqref{gapnotat}, for all $f:\O_\L\mapsto\bbR$:
\begin{align}
  \label{bibli1}
&\pi_\Lambda\left(1_{\{\eta\subset \Lambda_1\}}\Var_{\L_1}(f)\right)
\nonumber \\
&\qquad \leq\gamma(L,n-1) \frac12\sum_{i=1}^{\ell_\L}\pi_{\L}
\big[p_{i,+} (\nabla_{i,+}f)^21_{\{\eta_i\leq h_{\L_1}\}} +  p_{i,-} (\nabla_{i,-}f)^21_{\{\eta_i\leq h_{\L_1}\}}
\big]
\,.
\end{align}
Similarly, %for any $\eta\in\O_\L$:
\begin{align}
  \label{bibli2}
&\pi_\Lambda\big(\Var_{\L_2}^\eta(f)\big)
\nonumber \\
&\qquad \leq \gamma(L,n-1) \frac12\sum_{i=1}^{\ell_\L}\pi_{\L}
\big[p_{i,+} (\nabla_{i,+}f)^21_{\{\eta_i\geq h_{\L_1} - h_I\}} +  p_{i,-} (\nabla_{i,-}f)^21_{\{\eta_i\geq h_{\L_1} - h_I\}}
\big]
\,.
\end{align}
From Lemma \ref{gap-block}, \eqref{bibli1} and \eqref{bibli2} we have
\begin{equation}\label{bibli3}
\Var_\Lambda(f)\le \big(1+ e^{-(3/2)^{n/4}}\big)\,\gamma(L,n-1)\big(\cE_\L(f,f) + \cE^{I}_\L(f,f)\big)
\end{equation}
where $\cE_\L(f,f)$ is the usual Dirichlet form in $\L$, while 
$$
\cE^{I}_\L(f,f):=\frac12\sum_{i=1}^{\ell_\L}\pi_{\L}
\big[p_{i,+} (\nabla_{i,+}f)^21_{\{ h_{\L_1} - h_I\leq \eta_i\leq h_{\L_1}\}} +  p_{i,-} (\nabla_{i,-}f)^21_{\{h_{\L_1} - h_I\leq \eta_i\leq h_{\L_1}\}}
\big]
\,.
$$
Since \eqref{bibli3} is valid for every one of the $s_n$ choices of the rectangles $\big(\L_1^{(j)},\L_2^{(j)}\big)_{j=1}^{s_n}$, one can average this estimate to obtain, with $I^{(j)}=\L_1^{(j)}\cap\L_2^{(j)}$:
$$
\Var_\Lambda(f)\le \big(1+ e^{-(3/2)^{n/4}}\big)\,
\gamma(L,n-1)
\big(\cE_\L(f,f) + s_n^{-1}\sum_{j=1}^{s_n}\cE^{I^{(j)}}_\L(f,f)\big).
$$
Since the overlaps $I^{(j)}$ are disjoint, one has the obvious bound 
$\sum_{j=1}^{s_n}\cE^{I^{(j)}}_\L(f,f)\leq\cE_\L(f,f)$. It follows that
$$
\gamma(L,n)\le \big(1+ e^{-(3/2)^{n/4}}\big)\big(1+ s_n^{-1}\big)\gamma(L,n-1).
$$
By construction, $s_n^{-1}=O((3/2)^{-n/5})$, so that $(1+ e^{-(3/2)^{n/4}})(1+ s_n^{-1}) \leq  1+\d_n$
for all $n$ large enough, with $\d_n=(3/2)^{-n/6}$. This concludes the proof of Proposition \ref{recursive}.

\smallskip

\begin{proof}[Proof of Lemma \ref{gap-block}]
The proof is similar to that of \cite[Proposition 4.4]{CMRT}. 
Let $G=\{\eta\in\O_\L:\;\eta\subset \L_1\}$. 
Note that the event $G$ is very likely to occur under the equilibrium
$\pi_\L$. In particular, from \eqref{eq:eq}, one has 
\begin{equation}
  \label{eq:blo01}
\epsilon:=\pi_\L(G^c)\leq C\,\exp{(-C^{-1} (3/2)^{n})}.
\end{equation}
Moreover, $G$ is very likely to occur even under the equilibrium $\pi_{\L_2}^\eta$, 
{\em uniformly} in $\eta\in\O_\L$. Indeed, by monotonicity, the smallest value of $\pi_{\L_2}^\eta(G)$
is achieved at configurations $\eta\in\O_\L$ such that $\eta_i> h_{\L_1}-h_I$ for all $i=1,\dots,\ell_\L$ and in that case $\pi_{\L_2}^\eta(G)$ is the  $\pi_{\L_2'}$-probability  in the rectangle $\L_2':=[1,\ell_\L]\times [0,h_{\L_2}]$ %with zero boundary conditions 
that the height does not exceed $h_I$ at any point. Therefore, using  
$h_I^2/\ell_\L\geq \Delta_n^2(3/2)^{-n} = (3/2)^{n/2}$ and the bounds \eqref{eq:eq} one has 
\begin{equation}
  \label{eq:blo1}
\d:=\max_\eta\pi_{\L_2}^{\eta}(G^c)\leq C\,\exp{(-C^{-1} (3/2)^{n/2})}.
\end{equation}
On the other hand, reasoning as in \eqref{lowloc1} (and using $h_{\L_2}\gg\sqrt{\ell_\L}$), 
$\d$ also satisfies \[\d\geq  C^{-1}\,e^{-C (3/2)^{n/2}},\] and therefore in particular 
$\d\geq \sqrt {\epsilon/(1-\epsilon)}$ for all
$n$ large enough. 

The infinitesimal generator of the block dynamics acts on functions $f:\O_\L\mapsto\bbR$ as
$$
\cL_{\rm block}f(\eta) =  1_G(\eta)\,
(\pi_{\L_1}(f) - f(\eta))+ \pi^{\eta}_{\L_2}(f)-f(\eta)\,.
$$
Let $\l$ be the spectral gap of the block dynamics and let $f$ be an
eigenfunction of $\cL_{\rm block}$ with eigenvalue
$-\l$. 
%Clearly, $\pi_\L(f)=0$. Moreover, $\pi_{\L_2}^{\eta}(f)=f(\eta)$
% if $\eta\in G$. %$\cC(\eta)=\emptyset$. 
Thus,  $f$ must satisfy, for all $\eta\in\O_\L$: 
\begin{equation}
  \label{eq:gap1}
(1 -\l) f(\eta) = 1_{G}(\eta)\left(\pi_{\L_1}(f)-f(\eta)\right) + \pi_{\L_2}^{\eta}(f).
\end{equation}
If we average both sides w.r.t.\ $\pi_{\L_1} = \pi_\L(\cdot\tc G)$ 
%and use $\pi_\L^w(f)=0$ 
we get
\begin{gather}
(1-\l)\pi_{\L_1}(f)=
%\frac{1}{(1-\l)}
\pi_{\L_1}\big(\pi_{\L_2}^{\eta}(f)\big)
%\nonumber\\
=%\frac{1}{(1-\l)}
\pi_{\L}\big(\pi_{\L_2}^{\eta}(f)\tc G\big)=
%\frac{1}{(1-\l)}
%\frac{1}{\pi_\L(G)}
(1-\epsilon)^{-1}{\rm Cov}_{\L}\big(1_G\,,\,\pi_{\L_2}^{\eta}(f)\big)
\label{eq:gap2}
\end{gather}
where ${\rm Cov}_{\L}(\cdot\,,\,\cdot)$ denotes the covariance
w.r.t.\ $\pi_\L$, and we have used $\pi_\L(\pi_{\L_2}^{\eta}(f)) = \pi_\L(f)=0$.

Assume now  $1-\l>0$, otherwise there is
nothing to be proved. Schwarz' inequality gives 
${\rm Cov}_{\L}\big(1_G\,,\,\pi_{\L_2}^{\eta}(f)\big)\leq \Var_\L(1_G)^{1/2}\|\pi_{\L_2}^\cdot(f)\|_\infty$. Therefore, from 
\eqref{eq:gap2}, one has 
\begin{equation}\label{blo2}
|\pi_{\L_1}(f)| \le \frac{1}{(1-\l)}\sqrt{\frac{\epsilon}{1-\epsilon}}\,
 \|\pi_{\L_2}^\cdot(f)\|_\infty\le
\frac{\d}{(1-\l)}\,\|\pi_{\L_2}^{\cdot}(f)\|_\infty,
\end{equation}
where we use $\d\geq \sqrt {\epsilon/(1-\epsilon)}$, for $n$ large
enough
and $\|\pi_{\L_2}^{\cdot}(f)\|_\infty=\sup_\eta|\pi_{\L_2}^{\eta}(f)|$.
On the other hand we can rewrite  \eqref{eq:gap1} as
\begin{equation*}
  f(\eta)= \frac{1_G(\eta)\ \pi_{\L_1}(f)}{1-\l +1_G(\eta)} + \frac{\pi_{\L_2}^{\eta}(f)}{1-\l+1_G(\eta)}.
\end{equation*}
Therefore, applying $\pi_{\L_2}^{\eta}$ to both sides, we get
\begin{gather*}
 \pi_{\L_2}^{\eta}(f) = \frac{\pi_{\L_2}^{\eta}(G) \pi_{\L_1}(f)}{2-\l} + 
 \frac{\pi_{\L_2}^{\eta}(G) \pi_{\L_2}^\eta(f)}{2-\l} + \frac{\pi_{\L_2}^{\eta}(G^c) \pi_{\L_2}^\eta(f)}{1-\l}.
\end{gather*}
Then, using \eqref{blo2}, one has
$$
  \|\pi_{\L_2}^{\cdot}(f)\|_\infty \leq  \left(\frac\d{1-\l}\left(\frac1{2-\l}+1\right) +\frac1{2-\l}\right)\|\pi_{\L_2}^{\cdot}(f)\|_\infty=:\chi(\d,\l)
  \,\|\pi_{\L_2}^{\cdot}(f)\|_\infty .
$$
The latter bound is possible only if: (a) $\pi_{\L_2}^{\eta}(f)\equiv 0$ or (b) $\chi(\d,\l)\geq 1$.
It is easy to exclude (a). Suppose in fact that
$\pi_{\L_2}^{\eta}(f)\equiv 0$. Then also $\pi_{\L_1}(f)=0$ by \eqref{blo2}, and
\eqref{eq:gap1} would reduce to $(1-\l +1_G)f=0$ which is possible iff
$f\equiv 0$ because of the assumption $1-\l>0$.
On the other hand it is straightforward to check that
$ \chi(\d,\l)\geq 1$ implies that %for some large enough constant $C>0$ one has $\l\geq 1-C\sqrt\d$ for all $\d\leq 1/C$. This shows that 
$1-\l=O(\sqrt \d)\leq e^{-(3/2)^{n/4}}$ for all $n$ large enough. 
The proof of the Lemma is complete.
\end{proof}

\subsection{Spectral gap lower bound without the wall} \label{nowall}
The proof of Proposition \ref{propgap} in the case without the wall is very similar. The analogue of 
Proposition \ref{recursive} is described as follows. Given a rectangle $\L=[1,\ell_\L]\times [-h_\L,h_\L]$,
write $\gap(\L)=\gap(\ell_\L,h_\L)$ for the spectral gap of the SOS model in $\L$ with zero boundary conditions at $0$ and $\ell_\L+1$. Set 
$$
\beta(L,n)=\max_{\L\in \bbG_{n,L}}\gap(\L)^{-1}\,,
$$
where $ \bbG_{n,L}$ stands for the set of rectangles $\L=[1,\ell_\L]\times [-h_\L,h_\L]$ such that 
$\ell_\L\leq L$ and $h_\L\leq (3/2)^n$. The decompositions presented in items 
(a)-(b)-(c) in Section \ref{sec:rec} can now be obtained as follows. 
Let $\L_1 = [1,\ell_\L]\times[-h_{\L_1},h_{\L_1}]$ and $\L_2$ be given by the union of two rectangles
$\L_2=\L_2^{b}\cup\L_2^t$, where the bottom rectangle is $\L_2^b = [1,\ell_\L]\times[-h_{\L},-h_{\L}+h_{\L_2}]$ while the top rectangle is $\L_2^t = [1,\ell_\L]\times[h_{\L}-h_{\L_2},h_{\L}]$. Note that if 
$h_{\L_1}+h_{\L_2} > h_\L$ then there are two overlap regions $I^b,I^t$. As in Section \ref{sec:rec}, if $\L\in \bbG_{n,L}$ is such that $h_\L>(3/2)^{n-1}$ then one can find at least $s_n$ decompositions of the form $\L=\L_1\cup\L_2$ where $\L_i$, $i=1,2$ are as above with $h_{\L_1}+h_{\L_2} - h_\L \geq 2\Delta_n$, with $\L_1\in\bbG_{n-1,L}$ and $\L_2^t,\L_2^b\in\bbF_{n-1,L}$, and such that all the overlaps $I=I^t\cup I^b$ corresponding to distinct decompositions are disjoint.
Let $\pi_\L$, $\pi_{\L_1}$ be now the equilibrium distribution in $\L$, $\L_1$ with zero boundary conditions and $\pi_{\L_2}^\eta$ be the distribution in the region $\L_2$ conditioned to the value of $\eta\in\L\setminus\L_2$.   
With this notation, it is not hard to check that the result of Lemma \ref{gap-block} remains true as it stands. Then, the same argument of Section \ref{sec:rec} proves that
$$
\beta(L,n)\leq (1+\d_n)\wt\beta(L,n-1),\quad 
$$
where we define $\wt\beta(L,n):=\max\{\beta(L,n),\gamma(L,n)\}$. Since we know the results for the constants $\gamma(L,n)$ (cf.\ Proposition \ref{recursive}), 
it is now simple to infer the desired conclusion: %on the constants $\wt\beta(n,L)$ as well.
%Therefore, we conclude 
$\wt\beta(L,n)\leq CL^2(\log L)^\alpha$. This is sufficient to end the proof in the case without the wall.

\section*{Acknowledgements}
This work has been carried out while FLT
was visiting the Department of
Mathematics of the University of Roma Tre under the ERC Advanced Research
Grant ``PTRELSS''. FLT acknowledges partial support by  ANR through grant SHEPI.

\end{document}